\documentclass[a4paper, 10pt]{amsart}
\usepackage{amsmath, amssymb, hyperref, color}

\theoremstyle{plain}
\newtheorem{theorem}{\bf Theorem}[section]
\newtheorem{proposition}[theorem]{\bf Proposition}
\newtheorem{lemma}[theorem]{\bf Lemma}
\newtheorem{corollary}[theorem]{\bf Corollary}
\newtheorem{conjecture}[theorem]{\bf Conjecture}

\theoremstyle{definition}
\newtheorem{example}[theorem]{\bf Example}

\newtheorem{remark}[theorem]{\bf Remark}

\newcommand{\N}{\mathbb N}
\newcommand{\Z}{\mathbb Z}
\newcommand{\R}{\mathbb R}
\newcommand{\Q}{\mathbb Q}

\newcommand{\LK}{\,[\![}
\newcommand{\RK}{]\!]}

 \DeclareMathOperator{\Int}{Int}
\DeclareMathOperator{\spec}{spec} \DeclareMathOperator{\supp}{supp}
\DeclareMathOperator{\Pic}{Pic} 
 
\DeclareMathOperator{\inn}{\mathsf{in}} \DeclareMathOperator{\mdeg}{\mathsf {m-deg}}
 \DeclareMathOperator{\Span}{Span}
\DeclareMathOperator{\Rad}{Rad}

\newcommand{\red}{{\text{\rm red}}}

\newcommand{\BF}{\text{\rm BF}}
\newcommand{\FF}{\text{\rm FF}}
\newcommand{\fin}{\text{\rm fin}}

\newcommand{\DP}{\negthinspace : \negthinspace}

\numberwithin{equation}{section}

\begin{document}

\title{On the arithmetic of monoids of ideals}

\author{Alfred Geroldinger and M. Azeem Khadam}

\address{University of Graz, NAWI Graz \\
Institute of Mathematics and Scientific Computing \\
Heinrichstra{\ss}e 36\\
8010 Graz, Austria}
\email{alfred.geroldinger@uni-graz.at,  azeem.khadam@uni-graz.at}
\urladdr{https://imsc.uni-graz.at/geroldinger, https://sites.google.com/view/azeemkhadam}

\thanks{This work was supported by the Austrian Science Fund FWF, Project P33499-N}



\begin{abstract}
We  study the algebraic and arithmetic structure of monoids of invertible ideals (more precisely, of $r$-invertible $r$-ideals for certain ideal systems $r$) of Krull and weakly Krull Mori domains. We also investigate monoids of all nonzero ideals of polynomial rings with at least two indeterminates over noetherian domains. Among others, we show that they are not transfer Krull but they share several arithmetic phenomena with Krull monoids having infinite class group and prime divisors in all classes.
\end{abstract}

\maketitle


\section{Introduction} \label{1}

Let $R$ be a (commutative integral) domain, $r$ be an ideal system on $R$, $\mathcal I_r (R)$ be the semigroup of nonzero $r$-ideals with $r$-multiplication, and $\mathcal I_r^* (R) \subset \mathcal I_r (R)$ be the subsemigroup of $r$-invertible $r$-ideals. As usual, we denote by $v$ the system of divisorial ideals and for the $d$-system of usual ring ideals we omit all suffices (i.e., $\mathcal I (R) = \mathcal I_d (R)$, and so on). Factoring ideals into finite products of special ideals (such as prime ideals, radical ideals, and more) is a central topic of multiplicative ideal theory. The monograph \cite{Fo-Ho-Lu13a} of Fontana, Houston, and Lucas shows the rich variations of this theme. Algebraic and arithmetic properties of ideal semigroups help to understand the multiplicative structure of the underlying domain. To mention some classical results, $R$ is a Dedekind domain if and only if $\mathcal I (R) = \mathcal I^* (R)$ if and only if $\mathcal I (R)$ is a factorial monoid, and $R$ is a Krull domain if and only if $\mathcal I_v^* (R)$ is  a factorial monoid. Recent progress in such directions can be found in the work by Anderson, Chang, Juett, Kim, Klingler,  Olberding, Reinhart, and others  (e.g., \cite{Ch-Ki11a, Ju12a, Re12a, HK-Ol-Re16a, An-Ju-Mo19a, Ju-Mo-Nd21, Ol-Re19a, Ol-Re20a, Kl-Om20a}).

In the present paper, we first study factorizations of  $r$-invertible $r$-ideals into multiplicatively irreducible $r$-ideals of weakly Krull Mori domains and, in particular,  of Krull domains. Clearly, monoids $\mathcal I_r^* (R)$ of $r$-invertible $r$-ideals are commutative cancellative monoids.  After some preparations in the setting of abstract monoids in Section \ref{3},  we show in Section \ref{4} that the monoid of invertible ideals of a weakly Krull Mori domain is a weakly Krull Mori monoid again and that a domain $R$ is Krull if and only if the monoid $\mathcal I_r^* (R)$ (with $r$-multiplication) is a Krull monoid (Theorems \ref{4.3} and \ref{4.5}).   Much is known about the arithmetic of weakly Krull monoids and, in particular, of Krull monoids. The arithmetic of the latter is uniquely determined by its class group and the distribution of prime divisors in the classes. We focus on two arithmetical properties, namely on the structure of unions of sets of lengths and on being fully elastic (definitions are recalled at the beginning of Section \ref{3}). Our arithmetic results on the monoids of invertible ideals (as given in Corollary \ref{4.4} and Proposition \ref{4.9}) are based on our understanding of their algebraic structure.

In Section \ref{5}, we study the semigroup of all nonzero ideals. In general, these semigroups are not cancellative, and for this reason only first steps have been made towards the understanding of their arithmetic. However, under natural ideal-theoretic assumptions they are unit-cancellative and even BF-monoids (Propositions \ref{2.1} and \ref{2.2}). In the last years, parts of the existing machinery of factorization theory, developed in the setting of cancellative monoids, was generalized to the setting of unit-cancellative monoids (for a first paper, see \cite{F-G-K-T17}). In Section \ref{2}, we introduce all the required arithmetical concepts in the setting of unit-cancellative monoids. The monoid of all nonzero ideals was studied for  orders in Dedekind domains with finite class group \cite{Br-Ge-Re20, Ge-Re19d, Ba-Ge-Re21c}. In this setting, monoids of all nonzero ideals share  arithmetical finiteness properties with  monoids of invertible ideals and, more generally, with Krull monoids having finite class group. Our main result  in Section \ref{5} deals with the monoid of nonzero ideals of polynomial rings $R$ with at least two variables over noetherian domains, and they  show a completely different behaviour. Theorem \ref{5.1} shows that $\mathcal I (R)$ is not transfer Krull and that factorizations in $\mathcal I (R)$ are as wild as possible. Indeed,  they share arithmetical phenomena with Krull monoids having infinite class group and prime divisors in all classes (see Theorem \ref{5.1},  Conjecture \ref{5.12}, Example \ref{5.13} and the preceding discussion). The methods, used in the proof of Theorem \ref{5.1}, stem from the theory of Gr\"obner bases in polynomial ideal theory.

\smallskip
\section{Background on the ideal theory and the arithmetic of monoids} \label{2}
\smallskip

We denote by $\N$ the set of positive integers and we set $\N_0 = \N \cup \{0\}$. For real numbers $a, b \in \R$, we let $[a, b ] = \{x \in \Z \colon a \le x \le b \}$ denote the discrete interval between $a$ and $b$. Let $A$ and  $B$ be sets. We use the symbol $A \subset B$ to mean that $A$ is contained in $B$ but may be equal to $B$.  Suppose that $A$ and $B$ are  subsets of $\Z$. Then $A+B = \{a+b \colon a \in A, b \in B \}$ denotes their sumset and the set of distances $\Delta (A) \subset \N$ is the set of all $d \in \N$ for which there is $a \in A$ such that $A \cap [a, a+d] = \{a, a+d\}$. If $A \subset \N$, then $\rho (A) = \sup A / \min A \in \Q_{\ge 1} \cup \{\infty\}$ denotes the elasticity of $A$, and we set $\rho ( \{0\})= 1$.

Let $H$ be a multiplicatively written commutative semigroup with identity element. We denote by $H^{\times}$ the group of invertible elements of $H$, and we say that $H$ is reduced if $H^{\times} = \{1\}$. An element $a \in H$ is said to be
\begin{itemize}
\item {\it cancellative} if $b, c \in H$ and $ab = ac$ implies that $b=c$, and

\item {\it unit-cancellative} if $a \in H$ and $a = au$ implies that $u \in H^{\times}$.
\end{itemize}
By definition, every cancellative element is unit-cancellative. The semigroup $H$ is said to be {\it cancellative} (resp. {\it unit-cancellative}) if every element $a \in H$ is cancellative (resp. unit-cancellative).

\medskip
\centerline{\it Throughout this paper, a monoid means a}
\centerline{\it commutative unit-cancellative semigroup with identity element.}
\medskip

For a set $P$, we denote by $\mathcal F (P)$ the free abelian monoid with basis $P$. Elements $a\in\mathcal F(P)$ are written in the form
\[
a=\prod_{p\in P} p^{\mathsf v_p(a)}\,,\quad\text{where $\mathsf v_p\colon\mathcal F(P)\to\N_0$ }
\]
is the $p$-adic valuation. We denote by $|a|=\sum_{p\in P}\mathsf v_p(a)\in\N_0$ the length of $a$ and by $\supp(a)=\{ p\in P\colon \mathsf v_p(a) > 0\}\subset P$ the support of $a$. Let $H$ be a monoid. A monoid  $H$ is cancellative if and only if it has a quotient group, which will be denoted by $\mathsf q (H)$. Let $H$ be a cancellative monoid. We denote by
\begin{itemize}
\item $H' = \{ x \in \mathsf q (H) \colon \ \text{there is $N \in \N$ such that $x^n \in H$ for all $n \ge N$} \} \subset \mathsf q (H)$ the {\it seminormalization} of $H$, and by

\item $\widehat H = \{ x \in \mathsf q (H) \colon \ \text{there is $c \in H$ such that $cx^n \in H$ for all $n \in \N$} \} \subset \mathsf q (H)$   the {\it complete integral closure} of $H$.
\end{itemize}
Then $H \subset H' \subset \widehat H \subset \mathsf q (H)$, and $H$ is called
\begin{itemize}
\item {\it seminormal} if $H = H'$  (equivalently, if $x \in \mathsf q (H)$ and $x^2, x^3 \in H$, then $x \in H$), and

\item {\it completely integrally closed} if $H = \widehat H$.
\end{itemize}
A submonoid $S \subset H$ is called {\it divisor-closed} if $a \in S$ and $b \in H$ with $b \mid a $ implies that $b \in S$. For a subset $S \subset H$, we denote by $\LK S \RK$ the smallest divisor-closed submonoid generated by $S$. Let $\varphi \colon H \to D$ be a monoid homomorphism to a cancellative monoid $D$. We set $H_{\varphi} = \{a^{-1}b \colon a, b \in H, \varphi (a) \mid_D \varphi (b) \}$ and we say that $\varphi$ is a {\it divisor homomorphism} if $a, b \in H$ and $\varphi (a) \mid \varphi (b)$ in $D$ implies that $a \mid b $ in $H$ (equivalently, $H_{\varphi} = H$).

\smallskip
\noindent
{\bf Ideal Theory of Monoids.} Our notation of ideal theory follows \cite{HK98}, but note that the monoids in this paper do not contain a zero element. An {\it ideal system} on a cancellative monoid $H$ is a map $r \colon \mathcal P (H) \to \mathcal P (H)$, where $\mathcal P (H)$ is the power set of $H$,  such that the following conditions are satisfied for all subsets $X, Y \subset H$ and all elements $a \in H$:
\begin{itemize}
\item $X \subset X_r$,

\item $X \subset Y_r$ implies that $X_r \subset Y_r$,

\item $aH \subset \{a\}_r$, and

\item $aX_r = (aX)_r$.
\end{itemize}
We say that $r$ is {\it finitary} if, for all $X \subset H$,  $X_r$ is the union of all $E_r$ over all finite subsets $E \subset X$.
As usual, the $v$-system denotes the system of divisorial ideals. The monoid $H$ is said to be
\begin{itemize}
\item a {\it Mori monoid} if it is cancellative and satisfies the ACC on divisorial ideals, and

\item a {\it Krull monoid} if it is a completely integrally closed Mori monoid.
\end{itemize}
Let $r$ be any ideal system on $H$. A subset $I \subset H$ is called an $r$-ideal if $I_r=I$, and $\mathcal I_r (H)$ is the set of nonempty $r$-ideals. Then $\mathcal I_r (H)$ together with $r$-multiplication (defined by $I \cdot_r J = (IJ)_r$ for all $I, J \in \mathcal I_r (H)$) is a reduced semigroup with identity element $H$. Let $\mathcal F_r (H)$ denote the semigroup of fractional $r$-ideals and $\mathcal F_r (H)^{\times}$ the group of $r$-invertible fractional $r$-ideals. Then $\mathcal I_r^* (H) = \mathcal F_r (H)^{\times} \cap \mathcal I_r (H)$ is the cancellative monoid of $r$-invertible $r$-ideals with $r$-multiplication and $\mathcal I_r^* (H) \subset \mathcal I_r (H)$ is a divisor-closed submonoid. If $q$ is a further ideal system on $H$ with $\mathcal I_q (H) \subset \mathcal I_r (H)$, then, by \cite[Theorem 12.1]{HK98},
\begin{equation} \label{inclusion1}
\mathcal F_r (H)^{\times} \subset \mathcal F_q (H)^{\times} \subset \mathcal F_v (H)^{\times} \quad \text{are subgroups and} \quad \mathcal I_r^* (H) \subset \mathcal I_q^{*} (H) \subset \mathcal I_v^{*} (H) \ \text{are submonoids} \,.
\end{equation}
The cokernel of the group homomorphism $\mathsf q (H) \to \mathcal F_r (H)^{\times}$, $a \mapsto aH$, is called the $r$-class group of $H$. It will be denoted by $\mathcal C_r (H)$ and written additively. Thus, if $I, J \in \mathcal F_r (H)^{\times}$, then $[I \cdot_r J] = [I] + [J] \in \mathcal C_r (H)$. We denote by $\mathfrak X (H)$ the set of all minimal nonempty prime $s$-ideals of $H$ and note that $\mathfrak X (H) \subset t$-$\spec (H)$.
We say that $H$ satisfies the $r${\it -Krull Intersection Theorem} if
\[
\bigcap_{n \ge 0} (I^n)_r = \emptyset \quad \text{for all} \ I \in \mathcal I_r (H) \setminus \{H\} \,.
\]

\smallskip
\noindent
{\bf Arithmetic of Monoids.}  Let $H$ be a monoid. An element $p \in H$ is said to be
\begin{itemize}
\item {\it irreducible} (an {\it atom}) if $p \notin H^{\times}$ and $p = ab$ with $a, b \in H$ implies that $a \in H^{\times}$ or $b \in H^{\times}$,

\item {\it primary} if  $p \notin H^{\times}$ and $p \mid ab$ with $a, b \in H$ implies that $p \mid a$ or $p \mid b^n$ for some $n \in \N$, and

\item {\it prime} if $p \notin H^{\times}$ and $p \mid ab$ with $a, b \in H$ implies that $p \mid a$ or $p \mid b$.
\end{itemize}
We denote by $\mathcal A (H)$ the set of atoms of $H$.
The free abelian monoid $\mathsf Z (H) = \mathcal F ( \mathcal A (H_{\red}))$ is the {\it factorization monoid} of $H$ and  $\pi \colon \mathsf Z (H) \to H_{\red}$, defined by $\pi (u) = u$ for all $u \in \mathcal A (H_{\red})$, denotes the {\it factorization homomorphism} of $H$. For $a \in H$,
\begin{itemize}
\item $\mathsf Z_H (a) = \mathsf Z (a) = \pi^{-1} (aH^{\times}) \subset \mathsf Z (H)$ is the {\it set of factorizations} of $a$,

\item $\mathsf L_H (a) = \mathsf L (a) = \{ |z| \colon z \in \mathsf Z (a) \} \subset \N_0$ is the {\it set of lengths} of $a$, and

\item $\mathcal L (H) = \{ \mathsf L (a) \colon a \in H \}$ is the {\it system of sets of lengths} of $H$.
\end{itemize}
If $S \subset H$ is a divisor-closed submonoid and $a \in S$, then $\mathsf Z (S) \subset \mathsf Z (H)$, $\mathsf Z_S (a) = \mathsf Z_H (a)$, and $\mathsf L_S (a) = \mathsf L_H (a)$. We say that $H$ is
\begin{itemize}
\item {\it atomic} if $\mathsf L (a)$ is nonempty for all $a \in H$,

\item {\it half-factorial} if $|\mathsf L (a)| = 1$ for all $a \in H$,

\item a \BF-{\it monoid}  if $\mathsf L (a)$ is finite and nonempty for all $a \in H$,

\item an \FF-{\it monoid} if $\mathsf Z (a)$ is finite and nonempty for all $a \in H$, and

\item {\it locally finitely generated} if $\LK a \RK_{\red} \subset H_{\red}$ is finitely generated for all $a \in H$.
\end{itemize}
Every Mori monoid $H$ is a BF-monoid and if, in addition, $r$-$\max(H)=\mathfrak{X}(H)$, then $H$ is of finite $r$-character (\cite[Theorems 2.2.5.1 and 2.2.9]{Ge-HK06a}).
Krull monoids are locally finitely generated and locally finitely generated monoids are FF-monoids (\cite[Proposition 2.7.8]{Ge-HK06a}).

The next lemma gathers  the properties of ideal semigroups needed in the sequel.

\smallskip
\begin{proposition} \label{2.1}
Let $H$ be a cancellative monoid and $r$ be an ideal system on $H$.
\begin{enumerate}
\item If $H$ is a Mori monoid, then
      \[
      \bigcap_{n \ge 0} (I^n)_r = \emptyset \quad \text{for all} \ I \in \mathcal I_r^* (H) \setminus \{H\} \,.
      \]

\item If $H$ is $r$-noetherian, then $(\mathcal I_r^* (H), \cdot_r)$ is a Mori monoid.

\item If $(\mathcal I_r^* (H), \cdot_r)$ is a Mori monoid, then $H$ is a Mori monoid.

\item If $r$ is finitary and $H$ satisfies the $r$-Krull Intersection Theorem, then $\mathcal I_r (H)$ is unit-cancellative and if, in addition, $H$ has finite $r$-character, then $\mathcal I_r (H)$ is a \BF-monoid.
\end{enumerate}
\end{proposition}

\begin{proof}
1. See \cite[Theorem 12.5]{HK98}.

2. and 3. follow from  \cite[Example 2.1]{Ge-Ha08a}.

4. See \cite[Lemma 4.1]{Ba-Ge-Re21c} and \cite[Section 4]{Ge-Re19d}.
\end{proof}

\smallskip
\noindent
{\bf Rings.} By a ring, we mean a commutative ring with identity element and by a domain, we mean a commutative integral domain with identity element. Let $R$ be a  ring. Then its multiplicative semigroup $R^{\bullet}$ of regular  elements is a cancellative monoid.  All arithmetic concepts introduced for monoids will be used for the monoids of regular elements of rings. Thus, we say that $R$ is atomic (factorial, and so on) if $R^{\bullet}$ has the respective property and we set $\rho (R) := \rho (R^{\bullet})$ and similarly for all arithmetical invariants. 
If $R$ is a $v$-Marot ring, then $R$ is a Mori ring resp. a Krull ring if and only if $R^{\bullet}$ is a Mori monoid resp. a Krull monoid (\cite[Theorem 3.5]{Ge-Ra-Re15c}).

Let $R$ be a domain. We denote by $\mathcal H (R)$ the monoid of nonzero principal ideals of $R$, and note that $\mathcal H (R) \cong R^{\bullet}_{\red}$.
Let $r$ be an ideal system on $R$. Then   $r$ restricts to an ideal system $r'$ on $R^{\bullet}$, whence for every subset $I \subset R$, we have $I_r = (I^{\bullet})_{r'} \cup \{0\}$. We use all ideal theoretic concepts introduced for monoids for domains.  In particular,  $\mathcal I_r (R)$ is the semigroup of nonzero $r$-ideals of $R$, $\mathcal I_r^* (R)$ is the subsemigroup of $r$-invertible $r$-ideals of $R$, and $\mathcal C_r (R) = \mathcal F_r (R)^{\times}/ \mathsf q ( \mathcal H (R))$ is the $r$-class group of $R$.

The usual ring ideals form a finitary ideal system (the $d$-system), and for these ideals we omit all suffices, whence $\mathcal I (R) = \mathcal I_d (R)$, and the $d$-class group $\mathcal C_d (R) = \mathcal F (R)^{\times}/\mathsf q (\mathcal H (R))$ is the Picard group $\Pic (R)$ of $R$.
Throughout this paper,  we suppose that $\mathcal I_r (R) \subset \mathcal I (R)$, whence
\begin{equation} \label{inclusion2}
\mathcal I^* (R) \subset \mathcal I_r^* (R) \subset \mathcal I_v^* (R) \,.
\end{equation}
If $R$ satisfies the $r$-Krull Intersection Theorem, then $\mathcal I_r (R)$ is unit-cancellative by Proposition \ref{2.1}.4.  If $R$ is a Mori domain,  then $(\mathcal I_v^* (R), \cdot_v)$ is a    Mori monoid by Proposition \ref{2.1}.2. If $R$ is a one-dimensional Mori domain, then $\mathcal I_v^* (R) = \mathcal I^* (R)$, $R$ has finite character by \cite[Lemma 3.11]{Ga-Ro19a}, $\mathcal I (R)$ is unit-cancellative by \cite[Corollary 4.4]{Ge-Re19d}, whence $\mathcal I (R)$ is a \BF-monoid by Proposition \ref{2.1}.4.
For $I \in \mathcal I (R)$, we define $\omega' (I) \in \N_0 \cup \{\infty\}$ to be the smallest $N$ having the following property:
\begin{itemize}
\item[] If $n \in \N$ and $J_1, \ldots, J_n \in \mathcal I (R)$ with $J_1 \cdot \ldots \cdot J_n \subset I$, then there exists a subset $\Omega \subset [1,n]$ such that $|\Omega| \le N$ and $\prod_{\lambda \in \Omega} J_{\lambda} \subset I$.
\end{itemize}
The invariant $\omega' (I)$ is closely related to a well-studied invariant $\omega (\cdot)$ where, in the definition, containment is replaced by divisibility (see, for example, \cite{F-G-K-T17, Ge-Ha08a}). Thus, for invertible ideals $I \in \mathcal I^* (R)$, we have $\omega (I) = \omega' (I)$.

Suppose that  $R$ is noetherian. Then the integral closure $\overline R$ is a Krull domain by the Theorem of Mori-Nagata. Since noetherian domains are Mori, they are BF-domains but they need not be FF-domains. An algebraic characterization of when noetherian domains are locally finitely generated is given in \cite[Theorem 1]{Ka20a}. The next proposition shows that $\mathcal I (R)$ is a \BF-monoid.

\smallskip
\begin{proposition} \label{2.2}
	Let $R$ be a noetherian domain and $I \in \mathcal I (R)$.
	\begin{enumerate}
		\item If $J \in \mathcal I (R)$ with $I \subset J$ and $J/I \cong R/P$ (as $R$-modules) for some prime ideal $P \subset R$, then $\omega' (I) \le \omega' (J)+1$.
		
		\item $\omega' (I) < \infty$.
		
		\item $\sup \mathsf L_{\mathcal I (R)} (I) \le \omega' (I)$. In particular, $\mathcal I (R)$ is a \BF-monoid.
	\end{enumerate}
\end{proposition}

\begin{proof}
	Since $R$ satisfies Krull's Intersection Theorem, $\mathcal I (R)$ is unit-cancellative by Proposition \ref{2.1}.4.
	
	1. Since $J/I \cong R/P$, we obtain  $PJ \subset I$ and  $P = \{s\in R \colon sr \in I \}$ for every $r \in J\setminus I$. If $J_1, \ldots, J_n \in \mathcal{I}(R)$ with $J_1 \cdot \ldots \cdot J_n \subset I$, then $J_1 \cdot \ldots \cdot J_n \subset J$. Therefore, there exists a subset $\Omega \subset [1,n]$ such that $|\Omega| \le \omega'(J)$ and $\prod_{\lambda \in \Omega} J_{\lambda} \subset J$. If $\prod_{\lambda \in \Omega} J_{\lambda} \subset I$, then the statement follows. If $\prod_{\lambda \in \Omega} J_{\lambda} \not \subset I$, then there exists $r \in R$ with $r \in \prod_{\lambda \in \Omega} J_{\lambda} \setminus I \subset J\setminus I$.  Thus, we have
	\[
	r\cdot \prod_{\lambda \in [1,n] \setminus \Omega} J_{\lambda} \subset J_1 \cdot \ldots \cdot J_n \subset I \,,
	\]
and so $\prod_{\lambda \in [1,n] \setminus \Omega} J_{\lambda} \subset P$. Therefore, there exists $\ell \in [1,n ] \setminus \Omega$ such that $J_\ell \subset P$, whence
	\[
	J_\ell \cdot \prod_{\lambda \in \Omega} J_{\lambda} \subset PJ \subset I\,.
	\]
	
	2.  We need the following module theoretic result: if $M$ is a nonzero finitely generated $R$-module and $N\subset M$ a submodule, then there exist a chain of submodules $N=N_0\subset N_1\subset \ldots \subset N_m = M$ and prime ideals $P_i \subset R$ such that $N_{i+1}/N_i \cong R/P_i$ for all $i\in [0,m-1]$ ( \cite[Theorem 6.4]{Ma97}). Thus since   $\omega'(R) = 0$, the claim follows by 1.
	
	3. Let $I, J_1, \ldots, J_n \in \mathcal{I}(R)$ with $I = J_1 \cdot \ldots \cdot J_n$. Then there exists a subset $\Omega \subset [1,n]$ such that $|\Omega| \le \omega'(I)$ and $J' = \prod_{\lambda \in \Omega} J_{\lambda} \subset I$. Setting $J'' = \prod_{\lambda \in [1,n] \setminus \Omega} J_{\lambda}$, we obtain that
	\[
	J' \subset I = J'J'' \subset J' \,,
	\]
	whence $J' = J'J''$. Therefore $J'' = R$,   $\Omega = [1,n]$, and $n = |\Omega| \le \omega'(I)$. Finally, the in particular statement follows now by 2.
\end{proof}

\smallskip
\section{On unions of sets of lengths and sets of elasticities} \label{3}
\smallskip

Let $H$ be a \BF-monoid. Then
\[
\Delta (H) = \bigcup_{L \in \mathcal L (H)} \Delta (L) \ \subset \N
\]
is the {\it set of distances} of $H$. Let $k \in \N_0$. If $H = H^{\times}$, then we set $\mathcal U_k (H) = \{k\}$, and otherwise we set
\[
\mathcal U_k (H) = \bigcup_{k \in L, \, L \in \mathcal L (H)} L \ \subset \N_0
\]
is the {\it union of sets of lengths} containing $k$. Then $\rho_k (H) = \sup \mathcal U_k (H)$ is the {\it $k$-th elasticity} of $H$ and
\[
\rho (H) = \sup \{ \rho (L) \colon L \in \mathcal L (H) \} = \lim_{k \to \infty} \frac{\rho_k (H)}{k}
\]
is the {\it elasticity} of $H$. We say that $H$ has {\it accepted elasticity} if there is $L \in \mathcal L (H)$ with $\rho (L) = \rho (H)$. By definition, $H$ is half-factorial if and only if $\Delta (H) = \emptyset$ if and only if $\rho (H)= 1$. If $H$ is not half-factorial, then $\min \Delta (H)= \gcd \Delta (H)$.

We start with a discussion on elasticities.
In \cite{Ch-Ho-Mo06, B-C-H-M06}, Chapman et al. initiated the study of the set $\{\rho (L) \colon L \in \mathcal L (H) \} \subset \Q_{\ge 1}$ of elasticities of all sets of lengths.
By definition, $H$ is half-factorial if and only if  $\{ \rho (L) \colon L \in \mathcal L (H)\} = \{1\}$. The reverse extremal case, namely when the set of elasticities is as large as possible,  found special attention.
We say that $H$ is {\it fully elastic} if for every rational number $q$ with $1 < q < \rho (H)$ there is an $L \in \mathcal L (H)$ such that $\rho (L) = q$.
Thus, by definition, every half-factorial monoid is fully elastic.    For a detailed study of sets of elasticities in the setting of locally finitely generated monoids we refer to \cite{Zh19a}.

Next we discuss the structure of sets of lengths and of their unions. To do so, we need the concept of almost arithmetic (multi) progressions.
Let $d \in \N$, $M \in \N_0$, and $\{0,d\} \subset \mathcal D \subset [0,d]$.  A subset $L \subset \Z$ is called an
\begin{itemize}
\item {\it almost arithmetic multiprogression} (AAMP) with difference $d$, period $\mathcal D$, and bound $M$ if
      \[
      L = y + (L' \cup L^* \cup L'') \subset y + \mathcal D + d \Z \,,
      \]
      where $y \in \Z$ is a shift parameter, $L^*$ is finite nonempty, with $\min L^* = 0$ and $L^* = (\mathcal D + d \Z) \cap [0, \max L^*]$, $L' \subset [-M,-1]$, and $L'' \subset \max L^* + [1,M]$, and

\item {\it almost arithmetic progression} (AAP) with difference $d$ and bound $M$ if $L$ can be written in the form
      \[
      L = y + (L' \cup L^* \cup L'') \subset y + d \Z \,,
      \]
      where $y \in \Z$ is a shift parameter, $L' \subset [-M,-1]$, $L^*$ is a nonempty arithmetic progression with difference $d$ and $\min L^* = 0$,  $L'' \subset \max L^* + [1,M]$ if $L^*$ is finite, and $L'' = \emptyset$ if $L^*$ is infinite.
\end{itemize}
Note that AAPs need not be finite, whereas AAMPs are finite. Moreover, an AAMP with period $\mathcal D = \{0,d\}$ is an AAP with difference $d$.
We say that $H$ satisfies the
\begin{itemize}
\item {\it Structure Theorem for Sets of Lengths} if there are $M \in \N_0$ and a finite nonempty set $\Delta \subset \N$ such that every $L \in \mathcal L (H)$ is an AAMP with difference $d \in \Delta$ and bound $M$, and

\item {\it Structure Theorem for Unions} if  there are $d \in \N$ and $M \in \N_0$ such that $\mathcal U_k (H)$ is an AAP with difference $d$ and bound $M$ for all sufficiently large $k \in \N$.
\end{itemize}
We refer to \cite[Chapter 4.7]{Ge-HK06a} for background on the Structure Theorem for Sets of Lengths and to  \cite{Ga-Ge09b, F-G-K-T17,  Tr19a, Zh20a}, \cite[Theorem 6.6]{Bl-Ga-Ge11a}, \cite[Proposition 4.9]{Ch-Go-Go20}, \cite[Theorem 1.1]{Ba-Sm18}, \cite[Theorem 3.9]{Ba-Sa20a}, \cite[Theorem 5.4]{Oh-Zh20a}, \cite[Theorem 5.5]{Oh20a} for background and recent progress on the Structure Theorem for Unions.

The next lemma gathers  simple properties of unions of sets of lengths which we need in the sequel.

\smallskip
\begin{lemma} \label{3.1}
Let $H$ be a \BF-monoid and $k, \ell \in \N$.
\begin{enumerate}
\item $\mathcal U_k (H) + \mathcal U_{\ell} (H) \subset \mathcal U_{k+\ell} (H)$.

\item We have $k \in \mathcal U_{\ell} (H)$ if and only if $\ell \in \mathcal U_k (H)$.

\item If $\mathcal U_2 (H) = \N_{\ge 2}$, then $\mathcal U_k (H) = \N_{\ge 2}$ for all $k \ge 2$.
\end{enumerate}
\end{lemma}

\begin{proof}
1. and 2. follow immediately by the definitions.

3. Suppose that $\mathcal U_2 (H) = \N_{\ge 2}$. We proceed by induction on $k$. Let $k \ge 2$ and suppose that $\mathcal U_k (H) = \N_{\ge 2}$. Then 1. implies that
\[
\N_{\ge 3} = 1 + \N_{\ge 2} = 1 + \mathcal U_k (H) \subset \mathcal U_{k+1} (H) \,.
\]
Since $k+1 \in \mathcal U_2 (H)$, 2. implies that $2 \in \mathcal U_{k+1} (H)$. Thus, we obtain that $\mathcal U_{k+1} (H) = \N_{\ge 2}$.
\end{proof}

\smallskip
In \cite{B-C-C-K-W06} it was proved that every commutative cancellative monoid having a prime element and accepted elasticity is fully elastic. We generalize this result.

\smallskip
\begin{proposition} \label{3.2}
Let $H$ be a  \BF-monoid. Consider the following  conditions.
\begin{itemize}
\item[(a)] There are submonoids $H_1$ and $H_2$ such that $H = H_1 \times H_2$, where $H_1$ is half-factorial but not a group. This holds true in particular if $H$ has a cancellative prime element.

\item[(b)] For every $q \in \Q$ with $1 < q < \rho (H)$, there is an element $c \in H$ such that
      \[
      \rho (\mathsf L (c^k) ) = \rho (\mathsf L (c) ) > q \quad \text{for all $k \in \N$} \,.
      \]
      This holds true in particular if $H$ has accepted elasticity.

\item[(b')] For every $m \in \N_{\ge 2}$ there is $L \in \mathcal L (H_2)$ with $\min L = 2$ and $\max L = m$.
\end{itemize}
If Conditions (a) and (b) or Conditions (a) and (b') hold, then $H$ is fully elastic.
\end{proposition}

\begin{proof}
If $\rho (H)=1$, then Condition (b) holds trivially, $H$ is half-factorial, and hence it is fully elastic. From now on we suppose that $\rho (H) > 1$.

1.  We first check the two in particular statements stated in (a) and (b).

(i) If  $p \in H$ is a cancellative prime element of $H$, then
\[
H = \mathcal F (\{p\}) \times T \,, \text{where} \quad T = \{ a \in H \colon p \nmid a \} \,.
\]
If $a \in H$ and $k \in \N$, then
\[
\max \mathsf L (a^k) \ge k \max \mathsf L (a) \quad \text{and} \quad \min \mathsf L (a^k) \le k \min \mathsf L (a) \,,
\]
whence
\begin{equation} \label{calc1}
\rho ( \mathsf L (a^k) ) = \frac{\max \mathsf L (a^k)}{\min \mathsf L (a^k)} \ge \frac{k \max \mathsf L (a)}{k \min \mathsf L (a)} = \rho (\mathsf L (a) ) \,.
\end{equation}

(ii) Suppose that $H$ has accepted elasticity, say  $\rho (\mathsf L (a)) = \rho (H)$. By \eqref{calc1}, we infer that $\rho ( \mathsf L (a^k) ) = \rho (\mathsf L (a)) = \rho (H)$ for all $k \in \N$ and, if $\rho ( \mathsf L (a^k) ) = \rho (\mathsf L (a))$ for some $k \in \N$, then
\begin{equation} \label{calc2}
\max \mathsf L (a^k) = k \max \mathsf L (a) \quad \text{and} \quad \min \mathsf L (a^k) = k \min \mathsf L (a) \,.
\end{equation}

(iii) Suppose that (a) and (b) hold. To show that $H$ is fully elastic, we choose an atom $p \in H_1$.
Then for every $a = a_1a_2 \in H=H_1 \times H_2$ and every $k \in \N$, we have
\begin{equation} \label{calc3}
\mathsf L_H (p^ka) = \mathsf L_{H_1} (p^ka_1) + \mathsf L_{H_2}(a_2) = k+ \mathsf L_{H_1}(a_1) + \mathsf L_{H_2} (a_2) = k + \mathsf L_H (a) \,.
\end{equation}
Let $q \in \mathbb Q$ with $ 1 < q < \rho (H)$, and let $c \in H$ with
\[
\rho (\mathsf L (c^k) ) = \rho (\mathsf L (c) ) > q \quad \text{for all $k \in \N$} \,.
\]
We set
\[
q = \frac{r}{s} \quad \text{with} \quad r, s \in \N \,,
\]
\[
i = r - s, \quad  j = s \max \mathsf L (c) - r \min \mathsf L (c) , \quad \text{and} \quad b = c^i p^j \,.
\]
Then, by \eqref{calc2} and \eqref{calc3},
\[
\max \mathsf L (b) = j + i \max \mathsf L (c) \quad \text{and} \quad \min \mathsf L (b) = j + i \min \mathsf L (c) \,.
\]
Putting all together we obtain that
\[
\begin{aligned}
\rho ( \mathsf L (b) ) & = \frac{\max \mathsf L (b)}{\min \mathsf L (b)} = \frac{j + i \max \mathsf L (c)}{j + i \min \mathsf L (c)} \\
 & = \frac{(r-s)\max \mathsf L (c) + s \max \mathsf L (c) - r \min \mathsf L (c)}{(r-s)\min \mathsf L (c) + s \max \mathsf L (c) - r \min \mathsf L (c)} \\
 & = \frac{r (\max \mathsf L (c) - \min \mathsf L (c))}{s ( \max \mathsf L (c) - \min \mathsf L (c) )} \\
 & = \frac{r}{s} = q \,.
\end{aligned}
\]

2. Suppose that (a) and (b') hold and let $q \in \Q$ with $1 < q < \rho (H)$. Then there are $r, s \in \N$ such that $q = r/s$. By assumption, there is $a_2 \in H_2$ such that $\min \mathsf L (a_2) = 2$ and $\max \mathsf L (a_2) = r-s+2$. We choose an atom $u \in H_1$ and define
\[
b = u^{s-2} a_2 \,.
\]
Then
\[
\rho (\mathsf L (b)) = \frac{\max \mathsf L (b)}{\min \mathsf L (b)} = \frac{\max \mathsf L (u^{s-2}) + \max \mathsf L (a_2)}{\min \mathsf L (u^{s-2}) + \min \mathsf L (a_2)} = \frac{(s-2)+(r-s+2)}{(s-2)+2} = q \,. \qedhere
\]
\end{proof}

\smallskip
\begin{example} \label{3.3}~

1. Consider the additive monoid $H_2 = ( \N^2 \cup \{(0,0)\}, +) \subset ( \N_0^2 , +)$. For every $m \ge 2$, $(m-1,1), (1, m-1) \in \mathcal A (H_2)$ and we have
\[
(m,m) = (m-1,1) + (1, m-1) = \underbrace{(1,1) + \ldots + (1,1)}_{m-\text{times}} \,,
\]
whence $\min \mathsf L ( (m,m) ) = 2$ and $\max \mathsf L ( (m,m)) = m$. Thus, Condition (b') of Proposition \ref{3.2} is satisfied. Moreover, $\mathcal U_2 (H) = \N_{\ge 2}$, whence $\mathcal U_k (H) = \N_{\ge 2}$ for all $k \ge 2$ by Lemma \ref{3.1}. The forthcoming Proposition \ref{4.2} shows that $H_2$ is not fully elastic (because $H_2$ is strongly primary). However, if $H_1$ is any half-factorial monoid, then $H_1 \times H_2$ is fully elastic by Proposition \ref{3.2}.

2. In Section \ref{5}, we show that the monoid  of nonzero ideals of a polynomial ring with at least two variables also satisfy Conditions (a) and (b') of Proposition \ref{3.2} (Theorem \ref{5.1}).
\end{example}

\smallskip
\section{On monoids of invertible ideals of weakly Krull domains} \label{4}
\smallskip

In this section  we study the algebraic structure of monoids of invertible ideals and we derive some consequences for  their arithmetic.  Our focus will be on weakly Krull domains and Krull domains.
We start with a result in the setting of $r$-invertible $r$-ideals. Then our discussion is divided into four subsections, namely on weakly Krull domains \ref{4.a}, Krull domains \ref{4.b}, transfer Krull monoids \ref{4.c}, and on the arithmetic of transfer Krull monoids \ref{4.d}.

Let $H$ be a cancellative monoid and $r$ be an ideal system on $H$.
An ideal $I \in  \mathcal F_r (H)$ is called an {\it $r$-cancellation ideal} if whenever $I \cdot_r J_1 = I \cdot_r J_2 $ for  $J_1, J_2 \in \mathcal I_r (H)$, we have $J_1 = J_2$. It is easily seen
that $I$ is an $r$-cancellative if and only if whenever $I \cdot_r J_1 \subset I \cdot_r J_2$ for all $J_1, J_2 \in \mathcal I_r (H)$, we have $J_1 \subset J_2$. All $r$-invertible ideals (whence all principal ideals) are $r$-cancellation ideals, whence $\mathcal I_r^* (H)$ is a cancellative monoid.  A divisorial ideal is $v$-invertible if and only if it is $v$-cancellative (\cite[Chapter 13.4]{HK98}). Let $R$ be a domain. A nonzero ideal $I$ of  $R$ is a cancellation ideal if and only if $I$ is locally  principal (\cite{An-Ro97} and \cite{Ga13a}), and $\mathcal I (R)$ is a cancellative monoid if and only if $R$ is almost Dedekind (\cite[Theorem 23.2]{HK98}).  If $R$ is a Mori domain, then every nonzero locally principal ideal is invertible (\cite[Corollary 1]{An-Za11a}). If $p \in R$ is prime, then $pR$ is a cancellative prime element of $\mathcal I^* (R)$ and of $\mathcal I (R)$. Moreover, if $R$ is noetherian, then $p$ is also a prime element of $\overline R$ (\cite[Lemma 4.7]{Du-On08a}; for more on prime elements in noetherian domains, we refer to \cite{Da78a}).

\smallskip
\begin{theorem} \label{4.1}
Let $H$ be a cancellative monoid and let $r$ be an $r$-noetherian ideal system of $H$.

\begin{enumerate}
\item If $\mathcal I_r^* (H)$  has a  prime element, and if either $\mathcal I_r^* (H)$ has accepted elasticity or is locally finitely generated, then $\mathcal I_r^* (H)$ is fully elastic.

\item Suppose that   $H$ has finite $r$-character and satisfies the $r$-Krull Intersection Theorem.   If $H$  has an $r$-invertible prime $r$-ideal, and if either $\mathcal I_r (H)$ or all divisor-closed submonoids generated by one element have accepted elasticity, then $\mathcal I_r (H)$ is fully elastic.
\end{enumerate}
\end{theorem}

\begin{proof}
In both cases we verify the assumptions of Proposition \ref{3.2}.

1. Since $r$ is $r$-noetherian,   $\mathcal I_r^* (H)$ is a Mori monoid by Proposition \ref{2.1}.2 and hence it is a BF-monoid. Let $P \in \mathcal I_r^* (H)$ be a prime element. Since $\mathcal I_r^* (H)$ is a cancellative monoid, $P$ is a cancellative prime element.

Let $q$ be a rational number with $1 < q < \rho ( \mathcal I^* (H) )$. If $\mathcal I_r^* (H)$ has accepted elasticity, then there is $J \in \mathcal I_r^* (H)$ such that $\rho ( \mathsf L (J) ) =  \rho ( \mathcal I_r^* (H) )$, whence
\[
\rho (\mathsf L (J^k) ) = \rho (\mathsf L (J) ) > q \quad \text{for all $k \in \N$} \,.
\]
Now suppose that $\mathcal I_r^* (H)$ is locally finitely generated. We choose an ideal $I \in \mathcal I_r^* (H)$ with $\rho ( \mathsf L (I) ) > q$. We consider the divisor-closed submonoid $S = \LK I \RK \subset \mathcal I_r^* (H)$. Then $S$ is a finitely generated monoid. Thus, by \cite[Theorem 3.1.4]{Ge-HK06a},  there is an ideal $J \in S$ with $\rho (S) = \rho ( \mathsf L_S (J))  \ge \rho ( \mathsf L_S (I) ) = \rho ( \mathsf L_{\mathcal I_r^*(H)} (I) ) > q$. Therefore, we obtain that
\[
\rho (\mathsf L (J^k) ) = \rho (\mathsf L (J) ) > q \quad \text{for all $k \in \N$} \,.
\]
Thus, Conditions (a) and (b) of Proposition \ref{3.2} are satisfied, whence the assertion follows.

2. Since $H$ is $r$-noetherian, $r$ is a finitary ideal system by \cite[Theorem 3.5]{HK98}. Thus $\mathcal I_r (H)$ is a \BF-monoid by Proposition \ref{2.1}.4. Let $P$ be an $r$-invertible prime $r$-ideal. Then $P$ is a cancellative prime element of $\mathcal I_r (H)$. Arguing as in 1., we obtain an $r$-ideal $J$ such that
\[
\rho (\mathsf L (J^k) ) = \rho (\mathsf L (J) ) > q \quad \text{for all $k \in \N$} \,.
\]
Thus, Conditions (a) and (b) of Proposition \ref{3.2} are satisfied, whence the assertion follows.
\end{proof}

\smallskip
\subsection{Weakly Krull domains.}  \label{4.a}
In this subsection, we consider weakly Krull domains. We start with the local case and for this we
need the concept of primary monoids. Let $H$ be a cancellative monoid and $\mathfrak m = H \setminus H^{\times}$. Then $H$ is called
\begin{itemize}
\item {\it primary} if $H\ne H^{\times}$ and for all $a, b\in\mathfrak m$ there is $n\in\N$ such that $b^n\in aH$, and
\item {\it strongly primary} if $H\ne H^{\times}$ and for every $a\in\mathfrak m$ there is $n\in\N$ such that $\mathfrak m^n\subset aH$ (we denote by $\mathcal M(a)$ the smallest $n\in\N$ having this property).
\end{itemize}
Primary Mori monoids are strongly primary and strongly primary monoids are BF-monoids. The multiplicative monoid of nonzero elements of a domain is primary if and only if the domain is one-dimensional and local. If $R$ is a one-dimensional local Mori domain, then $\mathcal I_v^* (R) = \mathcal I^* (R) \cong R^{\bullet}$ and $R^{\bullet}$ is locally tame strongly primary (\cite[Corollary 3.10]{Ge-Ro20a}).

\smallskip
\begin{proposition} \label{4.2}
Let $H$ be a strongly primary monoid.
\begin{enumerate}
\item If $H$ is not half-factorial, then there is $\beta \in \Q_{>1}$ such that $\rho (L) \ge \beta$ for all $L \in \mathcal L (H)$ with $\rho (L) \ne 1$.  In particular, $H$ is fully elastic if and only if it is half-factorial.

\item If $H$ is locally tame, then $H$ satisfies the Structure Theorem for Sets of Lengths and the Structure Theorem for Unions.
\end{enumerate}
\end{proposition}

\begin{proof}
1. The first statement follows by \cite[Theorem 5.5]{Ge-Sc-Zh17b}. Since half-factorial monoids are fully elastic, the in particular statement holds.

2. This follows from \cite[Theorem 4.1]{Ge-Go-Tr21}.
\end{proof}

\smallskip
A family of monoid homomorphisms $\boldsymbol \varphi = (\varphi_p \colon H \to D_p)_{p \in P}$ is said to be
\begin{itemize}
\item of {\it finite character} if the set $\{p \in P \colon \varphi_p (a) \notin D_p^{\times} \}$ is finite for all $a \in H$, and

\item a {\it defining family} (for $H$) if it is of finite character and
      \[
      H = \bigcap_{p \in P} H_{\varphi_p} \,.
      \]
\end{itemize}
If $\boldsymbol \varphi$ is of finite character, then it induces a monoid homomorphism
\[
\varphi \colon H \to D = \coprod_{p \in P}(D_p)_{\red} \,, \quad \text{defined by} \quad \varphi (a) = \Big( \varphi_p (a)D_p^{\times} \Big)_{p \in P} \,,
\]
and  $\boldsymbol \varphi$ is a defining family if and only if $\varphi$ is a divisor homomorphism.

\smallskip
We recall the concept of weak divisor theories and weakly Krull monoids (\cite{HK95a}, \cite[Chapter 22]{HK98}). A monoid is said to be {\it weakly factorial} if it is cancellative and every nonunit is a finite product of primary elements. Every reduced weakly factorial monoid has a unique decomposition in the form
\[
D = \coprod_{p \in P} D_p, \quad \text{where $D_p \subset D$ are reduced primary submonoids}\,.
\]
Let $D$ be a reduced weakly factorial monoid as above. If $(a^{(i)})_{i \in I}$ is a family of elements $a^{(i)} \in D$ with components $a_p^{(i)}$, then $a$ is called a {\it strict greatest common divisor} of $(a^{(i)})_{i \in I}$, we write
\[
a = \wedge (a^{(i)})_{i \in I} \,,
\]
if the following two properties are satisfied for all $p \in P$:
\begin{itemize}
\item $a_p \mid a_p^{(i)}$ for all $i \in I$, and

\item $a_p = a_p^{(i)}$ for at least one $i \in I$.
\end{itemize}
A monoid homomorphism $\partial \colon H \to D$ is called a {\it weak divisor theory} if the following two conditions are satisfied:
\begin{itemize}
\item[(a)] $\partial$ is a divisor homomorphism and $D$ is reduced weakly factorial.

\item[(b)] For every $a \in D$, there are $a_1, \ldots, a_m \in H$ such that $a = \partial (a_1) \wedge \ldots \wedge \partial (a_m)$.
\end{itemize}
A monoid is said to be a {\it weakly Krull monoid} if it is cancellative and one of the following equivalent conditions is satisfied:
\begin{itemize}
\item $H$ has a weak divisor theory $\partial \colon H \to D$.

\item The family of embeddings $(\varphi_{\mathfrak p} \colon H \hookrightarrow H_{\mathfrak p})_{\mathfrak p \in \mathfrak X (H)}$ is a defining family for $H$.

\item $\varphi \colon H \to \coprod_{\mathfrak p \in \mathfrak X (H)} (H_{\mathfrak p})_{\red}$ is a weak divisor theory.
\end{itemize}
By the uniqueness of weak divisor theories, the {\it (weak divisor) class group} $\mathcal C (H) = \mathsf q (D) / \mathsf q ( \partial (H))$ depends on $H$ only and it is isomorphic to the $t$-class group $\mathcal C_t (H)$ of $H$ (\cite[Theorems 20.4 and 20.5]{HK98}). If $H$ is  weakly Krull Mori, then $\mathcal C (H) \cong \mathcal C_t (H) = \mathcal C_v (H)$. A monoid is weakly factorial if and only if it is weakly Krull with trivial class group. The localizations $H_{\mathfrak p}$ are primary for all $\mathfrak p \in \mathfrak X (H)$.

A domain $R$ is a weakly Krull domain if $R^{\bullet}$ is a weakly Krull monoid. If $R$ is a one-dimensional Mori domain, then $R$ is a  weakly Krull Mori domain,  $\mathcal I_v^* (R) = \mathcal I^* (R)$, and $\mathcal C_v (R) = \Pic (R)$ (\cite[Proposition 2.10.5]{Ge-HK06a}). In particular, orders in holomorphy rings of global fields are weakly Krull Mori domains and every class of their Picard group contains infinitely many invertible prime ideals (\cite[Corollary 2.11.16 and Proposition 8.9.7]{Ge-HK06a}). To mention higher-dimensional weakly Krull domains, recall that all Cohen-Macaulay domains are weakly Krull.  We mention  a  recent characterization of when monoid algebras are weakly Krull. Let $D$ be a domain with quotient field $K$ and let $S$ be a cancellative monoid with torsion-free quotient group $G = \mathsf q (S)$. Suppose that $G$ satisfies the ACC on cyclic subgroups. Then the monoid algebra $D[S]$ is weakly Krull if and only if $D$ is a weakly Krull domain satisfying the $G$-UMT property and $S$ is a weakly Krull monoid satisfying the $K$-UMT property (\cite[Theorem 3.7]{Fa-Wi22a}).
Monoid algebras, that are weakly Krull Mori and have height-one prime ideals in all classes, are studied in \cite{Fa-Wi22c}.

\smallskip
\begin{theorem} \label{4.3}
Let $R$ be a weakly Krull Mori domain. Then $\mathcal I_v^* (R)$ is a reduced weakly factorial Mori monoid.
The inclusion $\mathcal I^* (R) \hookrightarrow \mathcal I_v^* (R)$ is a weak divisor theory,  $\mathcal I^* (R)$ is a weakly Krull Mori monoid and its class group is isomorphic to $\mathcal C_v (R)/\Pic (R)$. If every class of $\mathcal C_v (R)$ contains at least one (resp. infinitely many) $\mathfrak p \in \mathfrak X (R)$, then every class of $\mathcal C_v ( \mathcal I^* (R))$ contains at least one (resp. infinitely many) $\mathfrak q \in \mathfrak X ( \mathcal I^* (R) )$.
\end{theorem}

\begin{proof}
By \cite[Proposition 5.3]{Ge-Ka-Re15a}, we have a monoid isomorphism
\begin{equation} \label{structure1}
\mathcal I_v^* (R) \longrightarrow \coprod_{\mathfrak p \in \mathfrak X (R)} (R_{\mathfrak p}^{\bullet})_{\red} \,,
\end{equation}
whence $\mathcal I_v^* (R)$ is a reduced weakly factorial Mori monoid (the Mori property follows from Proposition \ref{2.1}.2).
Since $R$ is a weakly Krull Mori domain, the inclusion $\mathcal H (R) \hookrightarrow \mathcal I_v^* (R)$ is a weak divisor theory. To verify that the inclusion $\mathcal I^* (R) \hookrightarrow \mathcal I_v^* (R)$ is a divisor homomorphism, let $I, J \subset R$ be invertible ideals such that $I \mid J$ in $\mathcal I_v^* (R)$. Then $I^{-1} J \in \mathcal F (R)^{\times} \cap \mathcal I_v^* (R) \subset \mathcal F (R)^{\times} \cap \mathcal I (R)  = \mathcal I^* (R)$. Thus, the inclusion $\mathcal I^* (R) \hookrightarrow \mathcal I_v^* (R)$ is a divisor homomorphism. This implies that $\mathcal I^* (R)$ is a Mori monoid by \cite[Proposition 2.4.4]{Ge-HK06a}. Since  $\mathcal H (R) \hookrightarrow \mathcal I_v^* (R)$ is a weak divisor theory, every $I \in \mathcal I_v^* (R)$ is a strict greatest common divisor of principal ideals and hence a strict greatest common divisor of invertible ideals. Therefore, $\mathcal I^* (R) \hookrightarrow \mathcal I_v^* (R)$ is a weak divisor theory and for the class group we have
\[
\mathsf q ( \mathcal I_v^* (R) ) / \mathsf q ( \mathcal I^* (R) ) = \mathcal F_v (R)^{\times} / \mathcal F (R)^{\times} \cong \Big(  \mathcal F_v (R)^{\times} / \mathsf q ( \mathcal H (R) ) \Big) \Big/ \Big(  \mathcal F (R)^{\times} / \mathsf q ( \mathcal H (R) ) \Big) = \mathcal C_v (R) / \Pic (R) \,.
\]
The claim on the distribution of prime divisors $\mathfrak q \in \mathfrak X ( \mathcal I^* (R))$ follows immediately from the above isomorphisms.
\end{proof}

\smallskip
Let $D$ be a weakly Krull monoid. If $H \subset D$ is a submonoid such that $H \hookrightarrow D$ is a divisor homomorphism and the class group $\mathsf q (D)/D^{\times}\mathsf q (H)$ is torsion, then $H$ is a weakly Krull monoid by \cite[Lemma 5.1]{Ge-Ka-Re15a}. This abstract result applies to the setting $\mathcal H (R) \hookrightarrow \mathcal I^* (R) \hookrightarrow \mathcal I_v^* (R)$, provided that the respective class groups are torsion. But, we did not check the general case.

\smallskip
Let $R$ be a weakly Krull Mori domain. Many aspects of the  arithmetic of $\mathcal I_v^* (R)$ have been studied in a variety of settings, from orders in quadratic number fields to seminormal weakly Krull domains to stable weakly Krull domains (see \cite[Theorem 5.8]{Ge-Ka-Re15a}, \cite[Theorem 5.8]{Ge-Sc-Zh17b}, \cite[Corollary 4.6]{Ge-Zh18a}, \cite[Theorem 1.1]{Br-Ge-Re20}, \cite[Theorem 5.13]{Ge-Re19d}, \cite[Theorem 5.10]{Ba-Ge-Re21c}). The following corollary characterizes when  - under some additional assumptions - $\mathcal I_v^* (R)$ is  fully elastic (compare with Proposition \ref{4.9}.1). For the sake of completeness and in order to compare it with Theorem \ref{5.1}, we also recall two results on the structure of sets of lengths and their unions.

\smallskip
\begin{corollary} \label{4.4}
Let $R$ be a weakly Krull Mori domain with nonzero conductor $(R \DP \widehat R)$.
\begin{enumerate}
\item $\mathcal I_v^* (R)$ satisfies the Structure Theorem for Sets of Lengths.

\item Suppose that $\widehat {R_{\mathfrak p}}^{\times}/R_{\mathfrak p}^{\times}$ is a torsion group for all $\mathfrak p \in \mathfrak X (R)$ and that $\mathcal I_v^* (R)$ has finite elasticity. Then $\mathcal I_v^* (R)$ satisfies the Structure Theorem for Unions, and it is fully elastic if and only if there is $\mathfrak q \in \mathfrak X (R)$ such that $R_{\mathfrak q}$ is half-factorial. If $\mathcal I_v^* (R)$ is not fully elastic, then $R$ is a one-dimensional semilocal Mori domain with $\mathcal C_v (R) = \boldsymbol 0$.
\end{enumerate}
\end{corollary}

\begin{proof}
1. See \cite[Theorem 7.4.3]{Ge-Zh20a}.

2. The additional assumptions imply that $\mathcal I_v^* (R)$ has accepted elasticity (\cite[Theorem 4.4.(ii)]{Ge-Zh18a}), whence it satisfies the Structure Theorem of Unions by \cite[Theorem 1.2]{Tr19a}.
We set $\mathfrak f = (R \DP \widehat R)$,  $\mathcal P^* = \{\mathfrak p \in \mathfrak X (R) \colon \mathfrak p \supset \mathfrak f \}$, $\mathcal P = \mathfrak X (R) \setminus \mathcal P^*$, and
\[
T = \coprod_{\mathfrak p \in \mathcal P^*} (R_{\mathfrak p}^{\bullet})_{\red} \,.
\]
By \eqref{structure1}, we  obtain that
\begin{equation} \label{structure2}
\mathcal I_v^* (R) \cong \coprod_{\mathfrak p \in \mathfrak X (R)} (R_{\mathfrak p}^{\bullet})_{\red} = \coprod_{\mathfrak p \in \mathcal P} (R_{\mathfrak p}^{\bullet})_{\red} \times T  \cong \mathcal F (\mathcal P) \times T \,.
\end{equation}
The localization  $R_{\mathfrak p}$ is a discrete valuation domain if and only if $\mathfrak p \in \mathcal P$. For all $\mathfrak p \in \mathcal P^*$,  $R_{\mathfrak p}^{\bullet}$ is a primary Mori domain, whence it is strongly primary.

(i) Let $\mathfrak q \in \mathfrak X (R)$ such that $R_{\mathfrak q}$ is half-factorial. Then
\[
\mathcal I_v^* (R) \cong \underbrace{(R_{\mathfrak q}^{\bullet})_{\red}}_{H_1} \times \underbrace{\coprod_{\mathfrak p \in \mathcal P \setminus \{\mathfrak q\}} (R_{\mathfrak p}^{\bullet})_{\red}}_{H_{2,1}} \times \underbrace{\coprod_{\mathfrak p \in \mathcal P^* \setminus \{\mathfrak q\}} (R_{\mathfrak p}^{\bullet})_{\red}}_{H_{2,2}} \,.
\]
If $\mathfrak X (R) = \{\mathfrak q\}$, then $\mathcal I_v^* (H)$ is half-factorial, whence it is fully elastic. Suppose that $\mathfrak X (R) \ne  \{\mathfrak q\}$. Then $H_2 := H_{2,1} \times H_{2,2}$ is a nontrivial monoid.
Since $H_{2,1}$ is free abelian, it has accepted elasticity. Since $\widehat {R_{\mathfrak p}}^{\times}/R_{\mathfrak p}^{\times}$ is a torsion group for all $\mathfrak p \in \mathfrak X (R)$ and since $\mathcal I_v^* (R)$ has finite elasticity, $(R_{\mathfrak p}^{\bullet})_{\red}$ has accepted elasticity by \cite[Lemma 4.1 and Theorem 4.4]{Ge-Zh18a} for all $\mathfrak p \in \mathcal P^* \setminus \{\mathfrak q\}$. Thus $H_2$ has accepted elasticity by \cite[Lemma 2.6]{Ge-Zh18a}.
Therefore,  $\mathcal I_v^* (R)$ is fully elastic by Proposition \ref{3.2}.

(ii) Suppose that $R_{\mathfrak q}$ is not half-factorial for all $\mathfrak q \in \mathfrak X (R)$. Then $\mathfrak X (R) = \mathcal P^*$ is finite, whence $\mathcal I_v^* (R) \cong T$ is a finite product of non-half-factorial strongly primary monoids, say $T = T_1 \times \ldots \times T_n$. Let $i \in [1,n]$. By Proposition \ref{4.2}, there is $\beta \in \Q_{>1}$ such that $\rho (L) \ge \beta$ for all $L \in \mathcal L (T_i)$ with $\rho (L) \ne 1$.
We set $\mathfrak m_i = T_i \setminus T_i^{\times}$ and we choose $a_i^* \in \mathfrak m_i$ with $|\mathsf L (a_i^*)|>1$. If $M = \max \{\mathcal M (a_1^*), \ldots, \mathcal M (a_n^*)\}$, then
\[
\mathfrak m_i^M \subset \mathfrak m_i^{\mathcal M (a_i^*)} \subset a_i^*T_i  \,.
\]
Now let $a \in T$ with $\rho ( \mathsf L (a) ) > 1$. Then $a = a_1 \cdot \ldots \cdot a_n$ and, after renumbering if necessary, we may assume that $\rho ( \mathsf L (a_i) ) > 1$ for all $i \in [1, m]$ and $\rho (\mathsf L (a_i) ) = 1$ for all $i \in [m+1, n]$ with $m \in [1,n]$. If $i \in [m+1,n]$, then $\rho (\mathsf L (a_i) ) = 1$ implies that $a_i \notin \mathfrak m_i^M$, whence $\mathsf L (a_i) = \{k_i\}$ for some $k_i \in [0, M-1]$.
Then
\[
\rho ( \mathsf L (a_1 \cdot \ldots \cdot a_m) ) = \frac{\max \mathsf L (a_1 \cdot \ldots \cdot a_m)}{\min \mathsf L (a_1 \cdot \ldots \cdot a_m)} = \frac {\max \mathsf L (a_1) + \ldots + \max \mathsf L (a_m)}{\min \mathsf L (a_1) + \ldots + \min \mathsf L (a_m)} \ge \beta \,,
\]
and
\[
\rho ( \mathsf L (a) ) = \frac{\max \mathsf L (a_1 \cdot \ldots \cdot a_m) + k_{m+1} + \ldots + k_n}{\min \mathsf L (a_1 \cdot \ldots \cdot a_m)+ k_{m+1} + \ldots + k_n} \,.
\]
Thus, there is $\beta^* \in \Q_{>1}$ such that $\rho ( \mathsf L (a) ) \ge \beta^*$ for all $a \in T$ with $\rho ( \mathsf L (a) ) > 1$, whence $T \cong \mathcal I_v^* (R)$ is not fully elastic.

(iii) Suppose that $\mathcal I_v^* (R)$ is not fully elastic. Since $v$-$\spec (R) = \mathfrak X (R)$ (\cite[Theorem 24.5]{HK98}), the equivalence of (i) and (ii) shows that $v$-$\spec (R)$ is finite. Thus, $v$-$\max (R) =\max (R)$ and $R$ is one-dimensional with $\mathcal C_v (R) = \boldsymbol 0$ by \cite[Propositions 2.10.4 and 2.10.5]{Ge-HK06a}.
\end{proof}

\smallskip
\subsection{Krull domains.} \label{4.b}
Let $H$ be a cancellative monoid. A {\it divisor theory} for $H$ is a weak divisor theory $\varphi \colon H \to D$, where $D$ is a free abelian monoid. The following statements are equivalent (\cite[Chapter 2.4]{Ge-HK06a}.
\begin{enumerate}
\item[(a)] $H$ is a Krull monoid (i.e., $H$ is a completely integrally closed Mori monoid).

\item[(b)] The map $\partial \colon H \to \mathcal I_v^* (H)$, defined by $a \mapsto aH$ for all $a \in H$, is a divisor theory.

\item[(c)] $H$ has a divisor theory $\varphi \colon H \to \mathcal F (P)$.

\item[(d)] There is a divisor homomorphism $\varphi \colon H \to D$, where $D$ is a factorial monoid.
\end{enumerate}
Let $H$ be a Krull monoid. Then there is a free abelian monoid $\mathcal F (P)$ such that the inclusion $H_{\red} \hookrightarrow \mathcal F (P)$ is a divisor theory. Then
\[
\mathcal C (H) = \mathsf q  (\mathcal F (P))/ \mathsf q (H_{\red})
\]
is called the {\it (divisor) class group} of $H$ and $G_P = \{[p] = p \mathsf q (H_{\red}) \colon p \in P \} \subset \mathcal C (H)$ is the set of classes containing prime divisors. By the uniqueness of divisor theories, $\mathcal C (H)$ and the set $G_P \subset \mathcal C (H)$ depend on $H$ only. In particular, $\mathcal C (H)$ is isomorphic to $\mathcal C_v (H)$.

\smallskip
\begin{theorem} \label{4.5}
Let $R$ be a  domain and $r$ be an ideal system on $R$ with $\mathcal I_r (R) \subset \mathcal I (R)$. Then the following statements are equivalent.
\begin{enumerate}
\item[(a)] $R$ is a Krull domain.

\item[(b)] The monoid of principal ideals $\mathcal H (R)$ is a Krull monoid.

\item[(c)] The monoid of invertible ideals $\mathcal I^* (R)$ is a Krull monoid.

\item[(d)] The monoid of $r$-invertible $r$-ideals $\mathcal I_r^* (R)$ is a Krull monoid.

\item[(e)] The monoid of $v$-invertible $v$-ideals $\mathcal I_v^* (R)$ is a Krull monoid.
\end{enumerate}
If these conditions hold, then the inclusion $\mathcal I_r^* (R) \hookrightarrow \mathcal I_v^* (R)$ is a divisor theory with class group $\mathcal C ( \mathcal I_r^* (R))$ being isomorphic to $\mathcal C_v (R)/\mathcal C_r (R)$.
Moreover, if every class of $\mathcal C_v (R)$ contains at least one prime divisor resp. infinitely many prime divisors, then the same is true for $\mathcal C (\mathcal I_r^* (R) )$.
\end{theorem}

\noindent
{\it Remark.} Clearly, there is a redundancy in the above formulation. If $r = d$ is the system of usual ideals, then $\mathcal I_r^* (R) = \mathcal I^* (R)$ and if $r=v$ is the system of divisorial ideals, then $\mathcal I_r^* (R) = \mathcal I_v^* (R)$. But, we want to emphasize these two important special cases.

\begin{proof}
A monoid $H$ is Krull if and only if the associated reduced monoid $H_{\red}$ is Krull and, clearly, $\mathcal H (R) \cong (R^{\bullet})_{\red}$. Thus (a) and (b) are equivalent by  \cite[Chapter 2]{Ge-HK06a}. Therefore, it remains to verify that (a) and (b) are equivalent to (d).
We have
\[
\mathsf q ( \mathcal I_r^* (R) ) = \mathcal F_r (R)^{\times} \quad \text{and} \quad \mathsf q ( \mathcal I_v^* (R) ) = \mathcal F_v (R)^{\times} \,.
\]
By Equation \eqref{inclusion1}, $\mathcal F_r (R)^{\times}$ is a subgroup of $\mathcal F_v (R)^{\times}$ and, in particular, $r$-ideal multiplication in $\mathcal F_r (R)^{\times}$ coincides with the $v$-multiplication. We continue with three assertions.

\begin{enumerate}
\item[{\bf A1.}\,] $\mathcal H (R) \hookrightarrow \mathcal I_r^* (R)$ is a  divisor homomorphism.

\item[{\bf A2.}\,] $\mathcal I_r^* (R) \hookrightarrow \mathcal I_v^* (R)$ is a  divisor homomorphism.

\item[{\bf A3.}\,] If $D$ is a Krull monoid and $H \hookrightarrow D$ is a divisor homomorphism, then $H$ is a Krull monoid.
\end{enumerate}

\noindent
{\it Proof of} \,{\bf A1}.\, If $a, b \in R^{\bullet}$ such that $aR \mid bR$ in $\mathcal I_r^* (R)$, then $bR = aR \cdot_r I$ for some $I \in \mathcal I_r^* (R)$, and $a^{-1}bR = I \subset R$ implies that $a \mid b$ in $R$.

\noindent
{\it Proof of} \,{\bf A2}.\, Let $I, J \in \mathcal I_r^* (R) $ such that $I$ divides $J$ in $\mathcal I_v^* (R)$. Then $I^{-1} \cdot_r J \in \mathcal F_r (R)^{\times} \cap \mathcal I_v^* (R) \subset \mathcal F_r (R)^{\times} \cap \mathcal I (R)  = \mathcal I_r^* (R)$. Thus, the inclusion $\mathcal I_r^* (R) \hookrightarrow \mathcal I_v^* (R)$ is a divisor homomorphism.

\noindent
{\it Proof of} \,{\bf A3}.\, If $D$ is Krull, then there is a divisor homomorphism $\varphi \colon D \to F$, where $F$ is a factorial monoid. Since the composition of divisor homomorphisms is a divisor homomorphism again, we obtain a divisor homomorphism $H \hookrightarrow D \to F$ from $H$ to a factorial monoid, whence $H$ is a Krull monoid.

Suppose that  (a) and (b) hold. Then $\mathcal I_v^* (R)$ is free abelian by \cite[Theorem 2.3.11]{Ge-HK06a} and hence Krull. Thus, $\mathcal I_r^* (R)$ is Krull by {\bf A2} and {\bf A3}, which means that (d) holds. Conversely, if (d) holds, then $\mathcal H (R)$ is Krull by {\bf A1} and {\bf A3}.

\smallskip
Now suppose that (a) -- (e) hold.
Since the inclusion $\mathcal H (R) \hookrightarrow \mathcal I_v^* (R)$ is a divisor theory, every $I \in \mathcal I_v^* (R)$ is a greatest common divisor of principal ideals and hence a greatest common divisor of $r$-invertible $r$-ideals. This, together with {\bf A2}, shows that the inclusion $\mathcal I_r^* (R) \hookrightarrow \mathcal I_v^* (R)$ is a divisor theory with class group
\[
\mathsf q ( \mathcal I_v^* (R) ) / \mathsf q ( \mathcal I_r^* (R) ) = \mathcal F_v (R)^{\times} / \mathcal F_r (R)^{\times} \cong \Big(  \mathcal F_v (R)^{\times} / \mathsf q ( \mathcal H (R) ) \Big) \Big/ \Big(  \mathcal F_r (R)^{\times} / \mathsf q ( \mathcal H (R) ) \Big) = \mathcal C_v (R) / \mathcal C_r (R) \,.
\]
The monoid $\mathcal I_v^* (R)$ is free abelian with basis $v$-$\spec (R)$. If every class of $\mathcal C_v (R)$ contains at least one resp. infinitely many prime $v$-ideals, then the same is true for the factor group $\mathcal C_v (R) / \mathcal C_r (R)$.
\end{proof}

Thus, if $R$ is a Krull domain, then the monoid $\mathcal I^* (R)$ of invertible ideals is a Krull monoid with class group isomorphic to $\mathcal C_v (R)/\Pic (R)$.  This factor group  shows a wide range of behavior. Daniel D. Anderson gave a characterization when the factor group   is a torsion group  and he showed that the factor group is trivial if and only if $R_{\mathfrak m}$ is factorial for all $\mathfrak m \in \max (R)$ (\cite[Theorems 3.1 and 3.3]{An82}; see also \cite{An88c}).
We consider monoid algebras that are Krull. For a domain $D$ and  a cancellative monoid $S$, the monoid algebra $D[S]$ has the following properties:
\begin{itemize}
\item $D[S]$ is Krull if and only if $D$ is Krull, $S$ is Krull, $\mathsf q (S)$ is torsion-free, and $S^{\times}$ satisfies the ACC on cyclic subgroups (this was first proved by Chouinard in \cite{Ch81}; see also \cite[Theorem 15.6]{Gi84}).

\item $D[S]$ is seminormal if and only if $D$ and $S$ are seminormal (\cite[Theorem 4.76]{Br-Gu09a}).
\end{itemize}

\smallskip
\begin{corollary} \label{4.6}
Let $D$ be a domain and $S$ be a cancellative monoid such that the monoid algebra $R = D[S]$ is Krull. Then $\mathcal I^* (R)$ is a Krull monoid with class group $\mathcal C ( \mathcal I^* (R)) \cong \mathcal C_v (R) / \Pic (R)$, $\mathcal C_v (R) \cong \mathcal C_v (D) \oplus \mathcal C_v (S)$, $\Pic (R) \cong \Pic (D)$, and every class of $\mathcal C_v ( \mathcal I^* (R))$ contains infinitely many prime divisors.
\end{corollary}

\begin{proof}
Since $D[S]$ is Krull, the previous remark implies that $D$ is a Krull domain and $S$ is a Krull monoid, whence $D$ and $S$ are seminormal. Thus, the natural map
\[
\Pic (D[S]) \longrightarrow \Pic (D)
\]
is an isomorphism by \cite[Corollary 1]{An88b} (note, since $D$ is completely integrally closed, $D$ is strongly quasinormal in the sense of \cite{An88b}).
We have $\mathcal C_v ( R) \cong \mathcal C_v (D) \oplus \mathcal C_v (S)$ by \cite[Corollary 16.8]{Gi84}.
Since every class of $\mathcal C_v (R)$ contains infinitely many prime divisors by \cite[Theorem]{Fa-Wi22b}, the same is true for the factor group $\mathcal C_v ( \mathcal I^* (R))$ by Theorem \ref{4.5}.
\end{proof}

Most arithmetical results, valid for weakly Krull Mori monoids $H$, are established under the additional assumption that $H$ has nonempty conductor $(H \DP \widehat H)$ to its complete integral closure. Now let $R$ be a weakly Krull Mori domain. We already know that $\mathcal I^* (R)$ is a weakly Krull Mori monoid. The next corollary shows that, if $R$ has nonzero conductor $(R \DP \widehat R)$, then also $\mathcal I^* (R)$ has nonempty conductor, whence all arithmetical results, valid for weakly Krull Mori monoids with nonempty conductor, also apply to $\mathcal I^* (R)$.

\smallskip
\begin{corollary} \label{4.7}
Let $R$ be a weakly Krull Mori domain with $(R \DP \widehat R) \ne \{0\}$. Then $\mathcal I_v^* (R)$ and $\mathcal I^* (R)$ are weakly Krull Mori monoids with  $( \mathcal I_v^* (R) \DP \widehat{\mathcal I_v^* (R)}) \ne \emptyset$ and with $( \mathcal I^* (R) \DP \widehat{\mathcal I^* (R)}) \ne \emptyset$.
\end{corollary}

\begin{proof}
$\mathcal I_v^* (R)$ and $\mathcal I^* (R)$ are weakly Krull Mori monoids by Theorem \ref{4.3}, whence it remains to prove the statements on the conductor.

(i) To prove the claim on $\mathcal I_v^* (R)$, we use the same notation as in the proof of Corollary \ref{4.4}. Thus, we set $\mathfrak f = (R \DP \widehat R)$,  $\mathcal P^* = \{\mathfrak p \in \mathfrak X (R) \colon \mathfrak p \supset \mathfrak f \}$, $\mathcal P = \mathfrak X (R) \setminus \mathcal P^*$, and by Equation \eqref{structure2} we have
\begin{equation} \label{structure3}
\mathcal I_v^* (R) \cong \coprod_{\mathfrak p \in \mathfrak X (R)} (R_{\mathfrak p}^{\bullet})_{\red}  \cong \mathcal F (\mathcal P) \times \coprod_{\mathfrak p \in \mathcal P^*} (R_{\mathfrak p}^{\bullet})_{\red} \,.
\end{equation}
Note that $\mathcal P^*$ is finite. We use the following simple facts on the complete integral closure of cancellative monoids $S_1, S_2$, and $S$.
\begin{itemize}
\item[(a)] If $S = S_1 \times S_2$, then $\widehat S = \widehat{S_1} \times \widehat{S_2}$. Thus, if $(S_i \DP \widehat{S_i}) \ne \emptyset$ for all $i \in [1,2]$, then $(S \DP \widehat S) \ne \emptyset$.

\item[(b)] If $(S \DP \widehat S) \ne \emptyset$ and $\mathfrak p \in \mathfrak X (S) \ne \emptyset$, then $(S_{\mathfrak p} \DP \widehat{S_{\mathfrak p}}) \ne \emptyset$.
\end{itemize}
Since $(R \DP \widehat R) \ne \{0\}$, it follows that $(R_{\mathfrak p}^{\bullet} \DP \widehat{R_{\mathfrak p}^{\bullet}}) \ne \emptyset$ for all $\mathfrak p \in \mathcal P^*$. Thus, by the isomorphism in \eqref{structure3} and by Property (a), it follows that $( \mathcal I_v^* (R) \DP \widehat{\mathcal I_v^* (R)}) \ne \emptyset$.

(ii) Next we show that $( \mathcal I^* (R) \DP \widehat{\mathcal I^* (R)}) \ne \emptyset$. By Theorem \ref{4.5}, the inclusion $\mathcal I^* (R) \hookrightarrow \mathcal I_v^* (R)$ is a divisor theory, whence $\mathcal I_v^* (R) \cap \mathsf q \big( \mathcal I^* (R) \big) = \mathcal I^* (R)$. By (i), there is
 $I \in ( \mathcal I_v^* (R) \DP \widehat{\mathcal I_v^* (R)}) \ne \emptyset$. Then there is $J \in \mathcal I_v^* (R)$ and $a \in R$ such that $I \cdot_v J = (IJ)_v = aR \in ( \mathcal I_v^* (R) \DP \widehat{\mathcal I_v^* (R)})$. Then
\[
(aR) \widehat{\mathcal I^* (R)} \subset (aR)\widehat{\mathcal I_v^* (R)} \cap \mathsf q \big( \mathcal I^* (R) \big) \subset \mathcal I_v^* (R) \cap \mathsf q \big( \mathcal I^* (R) \big) = \mathcal I^* (R) \,,
\]
whence $aR \in ( \mathcal I^* (R) \DP \widehat{\mathcal I^* (R)}) \ne \emptyset$.
\end{proof}

\smallskip
\subsection{Transfer Krull monoids} \label{4.c}

A monoid homomorphism  $\theta \colon H \to B$ is called a {\it transfer homomorphism} if it has the following properties:
      \begin{enumerate}
      \item[{\bf (T\,1)\,}] $B = \theta(H) B^\times$ \ and \ $\theta ^{-1} (B^\times) = H^\times$.

      \item[{\bf (T\,2)\,}] If $u \in H$, \ $b,\,c \in B$ \ and \ $\theta (u) = bc$, then there exist \ $v,\,w \in H$ \ such that \ $u = vw$, \ $\theta (v) \in bB^{\times}$, and  $\theta (w) \in c B^{\times}$.
      \end{enumerate}
Transfer homomorphisms allow to pull back arithmetical properties from $B$ to $H$. In particular, we have $\mathsf L_H (a) = \mathsf L_B ( \theta (a))$ for all $a \in H$, whence $\mathcal L (H) = \mathcal L (B)$. This implies that an element $a \in H$ is an atom of $H$ if and only if $\theta (a)$ is an atom of $B$.
A monoid $H$ is called {\it transfer Krull} if there are a Krull monoid $B$ and a transfer homomorphism $\theta \colon H \to B$.
A commutative ring $R$ is said to be transfer Krull if its monoid $R^{\bullet}$ of regular elements is a transfer Krull monoid.

If $H$ is half-factorial, then $\theta \colon H \to (\N_0, +)$, defined by $\theta (u) = 1$ for all $u \in \mathcal A (H)$ and $\theta ( \varepsilon ) = 0$ for every $\varepsilon \in H^{\times}$, is a transfer homomorphism, whence all half-factorial monoids are transfer Krull.
Furthermore, all Krull monoids are transfer Krull (with $\theta$ being the identity).
Since transfer homomorphisms preserve lengths of factorizations, all transfer Krull monoids are \BF-monoids but they need neither be $v$-noetherian nor completely integrally closed.  A list of transfer Krull monoids and domains that are not Krull can be found in  \cite[Example 5.4]{Ge-Zh20a}, and we refer to \cite{Ba-Re22a} for a systematic study of the transfer Krull property. On the other hand, here are some  monoids that are not transfer Krull.
\begin{itemize}
\item $\Int (\Z)$ is not transfer Krull, by \cite[Remark 12]{Fr13a}.

\item The monoid of polynomials having nonnegative integer coefficients is not transfer Krull, by \cite[Remark 54.]{Ca-Fa19a}.

\item The monoid of finite nonempty subsets of the nonnegative integers (with set addition as operation) is not transfer Krull, by \cite[Proposition 4.12]{Fa-Tr18a} (see the discussion after Conjecture \ref{5.12}).
\end{itemize}
Moreover,   let $R$ be a weakly Krull Mori domain. Then $\mathcal I_v^* (R)$ is transfer Krull if and only if it is half-factorial (in the local case this follows from Proposition \ref{4.2}.1 and Proposition \ref{4.9}; the general case is a simple consequence, see \cite[Proposition 7.3]{Ge-Zh20a} and also \cite[Theorem 5.9]{Ba-Ge-Re21c}).
For more results of this flavor, see the references given in the discussion before Corollary \ref{4.4}).

We continue with a simple lemma which we will use to show that the monoid of nonzero ideals over a polynomial ring with at least two variables is not transfer Krull (Theorem \ref{5.1}).

\smallskip
\begin{lemma} \label{4.8}
Let $H$ be a cancellative monoid and $r$ be an ideal system on $H$. Suppose there is a non-$r$-cancellative ideal $I \in \mathcal I_r (H)$, an ideal $J_1 \in \mathcal A ( \mathcal I_r (H))$, and a $J_2 \in \mathcal I_r (H) \setminus \mathcal A ( \mathcal I_r (H))$ such that $I \cdot_r J_1 = I \cdot_r J_2 $. Then $\mathcal I_r (H)$ is not a transfer Krull monoid.
\end{lemma}

\begin{proof}
Assume to the contrary that there are a Krull monoid $B$ and a transfer homomorphism $\theta \colon \mathcal I_r (H) \to B$. Then $\theta (I) \theta (J_1) = \theta (I) \theta (J_2)$, whence $\theta (J_1) = \theta (J_2)$ because $B$ is cancellative. Since $J_1$ is an atom of $\mathcal I_r (H)$, $\theta (J_1) = \theta (J_2)$ is an atom of $B$ and hence $J_2$ is an atom of $\mathcal I_r (H)$, a contradiction.
\end{proof}

\smallskip
\subsection{Arithmetic of transfer Krull monoids} \label{4.d}
The arithmetic of Krull monoids is determined by their class groups and  the distribution of prime divisors in the  classes. There is an abundance of literature on the arithmetic of Krull monoids (see \cite{Ge-HK06a} and the survey \cite{Sc16a}).
We briefly summarize some results valid not only for Krull monoids but more generally for  transfer Krull monoids, but we restrict for results on sets of lengths. This will allow us to compare them with the arithmetic of the monoids of ideals discussed in Section \ref{5}. In order to do so we recall the monoid of zero-sum sequences over an abelian group.

Let $G$ be an additive abelian group and $G_0 \subset G$ be a subset. An element $S = g_1 \cdot \ldots \cdot g_{\ell} \in \mathcal F (G_0)$, with $\ell \in \N_0$ and $g_1, \ldots, g_{\ell} \in G_0$, is called a sequence over $G_0$. Then $|S|=\ell \in \N_0$ is the length of $S$, $\sigma (S) = g_1 + \ldots + g_{\ell} \in G$ is the sum of $S$, and
\[
\mathcal B (G_0) = \{ T \in \mathcal F (G_0) \colon \sigma (T) = 0 \} \subset \mathcal F (G_0)
\]
is the {\it monoid of zero-sum sequences} over $G_0$. Since the inclusion $\mathcal B (G_0) \hookrightarrow \mathcal F (G_0)$ is a divisor homomorphism, $\mathcal B (G_0)$ is a Krull monoid. As usual, we set $\mathcal L (G_0) := \mathcal L (\mathcal B (G_0))$, $\rho (G_0) : = \rho ( \mathcal B (G_0))$, and $\mathcal U_k (G_0) := \mathcal U_k ( \mathcal B (G_0))$ for all $k \in \N$.

Let $B$ be a Krull monoid, $\varphi \colon B \to D = \mathcal F (P)$ be a divisor theory, and let $G_P = \{ [p] \colon p \in P\} \subset G$ denote the set of classes containing prime divisors. The map
\[
\boldsymbol \beta \colon B \to \mathcal B (G_P), \quad \text{defined by} \quad \boldsymbol \beta (a) = [p_1] \cdot \ldots \cdot [p_{\ell}] \,,
\]
where $\varphi (a) = p_1 \cdot \ldots \cdot p_{\ell}$ with $p_1, \ldots, p_{\ell} \in P$, is a transfer homomorphism.

Let $H$ be a transfer Krull monoid and $\theta_1 \colon H \to B$ be a transfer homomorphism to a Krull monoid $B$. If $G$ is an abelian group, $G_0 \subset G$ a subset, and $\theta_2 \colon B \to \mathcal B (G_0)$, then $\theta = \theta_2 \circ \theta_1 \colon H \to \mathcal B (G_0)$ is a transfer homomorphism from $H$ to the monoid of zero-sum sequences over $G_0$. In this case, we say that $H$ is a transfer Krull monoid over $G_0$. Since every Krull monoid has a transfer homomorphism onto a monoid of zero-sum sequences, every transfer Krull monoid has a transfer homomorphism to a monoid of zero-sum sequences. If $H$ is a Krull monoid with class group $G$ and every class contains at least one prime divisor, then $H$ is a transfer Krull monoid over the class group $G$.

\smallskip
\begin{proposition} \label{4.9}
Let $H$ be a transfer Krull monoid and let $\theta \colon H \to \mathcal B (G_0)$ be a transfer homomorphism, where $G_0 \subset G$ is a subset of an abelian group.
\begin{enumerate}
\item $H$ is fully elastic.

\item If $G_0$ is finite, then the elasticity $\rho (H) < \infty$, $H$ satisfies the Structure Theorem for Sets of Lengths as well as the Structure Theorem for Unions.

\item If $G_0$ contains an infinite abelian group, then for every finite subset $L \subset \N_{\ge 2}$, there is $a \in H$ such that $\mathsf L (a) = L$, whence $\mathcal U_k (H) = \N_{\ge 2}$ for all $k \ge 2$.
\end{enumerate}
\end{proposition}

\begin{proof}
1. This follows from \cite[Theorem 3.1]{Ge-Zh19a}.

2. Suppose that $G_0$ is finite. We have $\mathcal L (H) = \mathcal L (G_0)$, whence $\rho (H) = \rho (G_0)$, and $\mathcal U_k (H) = \mathcal U_k (G_0)$ for all $k \in \N$. Since $G_0$ is finite, $\mathcal B (G_0)$ is finitely generated, whence $\rho (G_0) < \infty$ by \cite[Theorem 3.1.4]{Ge-HK06a}. Furthermore, $H$ satisfies the Structure Theorem for Sets of Lengths by \cite[Chapter 4.7]{Ge-HK06a} and the Structure Theorem for Unions by \cite[Corollary 3.6 and Theorem  4.2]{Ga-Ge09b}.

3. Suppose that $G_0$ contains an infinite abelian group $G_1$. By \cite[Theorem 7.4.1]{Ge-HK06a}, every finite subset $L \subset \N_{\ge 2}$  lies in $\mathcal L (G_1)$, and hence $L \in \mathcal L (G_1) \subset \mathcal L (G_0) = \mathcal L (H)$. Since $\{2, k\} \in \mathcal L (H)$ for every $k \in \N_{\ge 2}$, it follows that $\mathcal U_2 (H) = \N_{\ge 2}$, whence $\mathcal U_k (H) = \N_{\ge 2}$ for every $k \ge 2$ by Lemma \ref{3.1}.
\end{proof}

Let $\theta \colon H \to \mathcal B (G_0)$ be as above. If $G_0$ is a finite abelian group, then there is a rich literature on invariants controlling the structure of sets of lengths (\cite{Sc16a}). The elasticity $\rho (H)$ can be finite even if $G_0$ is infinite (if $G$ is finitely generated, then \cite{Gr22a} offers a characterization of when $\rho (H)$ is finite, and if $\rho (H) < \infty$, then also the Structure Theorem for Unions holds). The unions $\mathcal U_k (H)$ are intervals if $G_0$ is a group, but they need not be intervals in general.
There are  Krull monoids that do not satisfy the  Structure Theorem for Unions (\cite[Theorem 4.2]{F-G-K-T17}), and there are Krull monoids that neither satisfy the Structure Theorem for Sets of Lengths nor does every finite subset $L \subset \N_{\ge 2}$ occur as a set of lengths.

\smallskip
\section{On the monoid of nonzero ideals of polynomial rings} \label{5}
\smallskip

The main goal of this section is to prove the  result given in Theorem \ref{5.1}. We start with a couple of remarks. Let $D$ be a noetherian domain, $n \ge 2$,  $S = (\N_0^n, +)$, and $R = D[S] = D [X_1, \ldots, X_n]$.  Then $D$ is Krull if and only if $D$ is integrally closed if and only if $R$ is Krull if and only if $\mathcal I^* (D)$ resp. $\mathcal I^* (R)$ are Krull (see Theorem \ref{4.5}). Furthermore, $\mathcal C (D)$ and $\mathcal C (R)$ are isomorphic, $D$ is factorial if and only if $R$ is factorial if and only if $\mathcal C (D)$ is trivial. If $D$ is factorial (for example, if $D$ is a field), then
\begin{equation} \label{invertible}
\{ aR \colon a \in R^{\bullet} \} = \mathcal I^* (R) = \mathcal I_v^* (R) = \mathcal I_v (R) \quad \text{and all these monoids are factorial.}
\end{equation}
In orders of Dedekind domains with finite class group, monoids of all nonzero ideals and monoids of invertible ideals have similar arithmetical properties (\cite{Br-Ge-Re20, Ge-Re19d, Ba-Ge-Re21c}).
In contrast to \eqref{invertible} and in contrast to orders in Dedekind domains, our conjecture (Conjecture \ref{5.12}) is that the arithmetic of the monoid $\mathcal I (R)$ is completely different from the arithmetic of $\mathcal I^* (R)$ and that it is as wild as it is for Krull monoids with infinite class group and prime divisors in all classes (see Proposition \ref{4.9}.3). The main result of this section (Theorem \ref{5.1}) is a first step towards this conjecture.

\smallskip
\begin{theorem} \label{5.1}
Let $R = D[X_1, \ldots, X_n]$ be the polynomial ring in $n \ge 2$ indeterminates over a domain $D$, and suppose that $\mathcal I (R)$ is a \BF-monoid.
\begin{enumerate}
\item $\mathcal I (R)$  is  neither transfer Krull nor locally finitely generated. Moreover, if  $D^{\times}$ is infinite, then $\mathcal I (R)$ is not an \FF-monoid.

\item $\mathcal U_k ( \mathcal I (R) ) = \N_{\ge 2}$ for all $k \ge 2$.

\item $\mathsf L_{\mathcal I (R) } ( \langle X_1, X_2 \rangle^k ) = [2, k]$ for all $k \ge 2$.

\item $\mathcal I (R)$ is fully elastic.
\end{enumerate}
\end{theorem}

\smallskip
We briefly discuss the assumption that $\mathcal I (R)$ is a \BF-monoid (made in Theorem \ref{5.1}, Lemma \ref{5.3}, Proposition \ref{5.10}, and Conjecture \ref{5.12}). If $R$ is noetherian or a one-dimensional Mori domain, then $\mathcal I (R)$ is a \BF-monoid (see Proposition \ref{2.2} and the discussion after Proposition \ref{2.1}). But the property, that $\mathcal I (R)$ is a \BF-monoid, seems to be much weaker than the above two assumptions. A crucial property in this context is Krull's Intersection Theorem, which guarantees that the semigroup $\mathcal I (R)$ is unit-cancellative.
We mention two further  classes of polynomial rings which satisfy Krull's Intersection Theorem (for more on the validity of Krull's Intersection Theorem, we refer to \cite{An-Ma-Ni76, HK98}).
\begin{itemize}
	\item [(i)] Let $D$ be a domain, $\overline D$ its integral closure (in the quotient field of $D$), and let $D^*$ be any domain with $D \subset D^* \subset \overline D$. The integral closure of $D[X_1, \ldots, X_n]$ equals $\overline D [X_1, \ldots, X_n]$, and we have
	\[
	D[X_1, \ldots, X_n] \subset D^*[X_1, \ldots, X_n] \subset \overline D[X_1, \ldots, X_n] \,.
	\]
	If  $D$ is noetherian, then $D[X_1, \ldots, X_n]$ is noetherian, whence $D^*[X_1, \ldots, X_n]$ satisfies Krull's Intersection Theorem by \cite[Theorem 5]{An75a}. Moreover, if $D$ is noetherian, then $\overline D$ is Krull and if $D$ is Krull, then $D[X_1, \ldots, X_n]$ is a Krull domain with class group isomorphic to $\mathcal C (D)$ and infinitely many prime divisors in all classes (compare with Corollary \ref{4.6}).
	
	\item [(ii)] If  the integral closure $\overline{R}$ of $R  = D[X_1, \ldots, X_n]$ in some field extension of the quotient field of $R$ is noetherian, then $R$ satisfies Krull's Intersection Theorem by \cite[Proposition 2.6]{He-La-Re21}.
\end{itemize}

\smallskip
We proceed in a series of lemmas.
Let $R$ be a domain and $X_1,X_2 \in R^{\bullet}$. For all $i \in \N$, we consider the following four families of  nonzero ideals of $R$:
\begin{itemize}
	\item [(i)] $\mathfrak{a}_i(X_1,X_2) := \langle X_1, X_2 \rangle^i$,
	
	\item [(ii)] $\mathfrak{b}_i(X_1,X_2) := \langle X_1^i, X_2^i \rangle$,
	
	\item [(iii)] $\mathfrak{c}_{2i+1}(X_1,X_2) := \langle \{X_1^{2i+1}, X_1^{2i}X_2\}\cup \{X_1^{2i-j}X_2^{j+1}~:~ j\in[1,2i+1] \ \text{is even} \} \rangle$, and
	
	\item [(iv)] $\mathfrak{c}_{2i}(X_1,X_2) := \langle \{X_1^{2i}, X_1^{2i-1}X_2\}\cup \{X_1^{2i-j}X_2^{j}~:~ j\in[1,2i] \ \text{is even}  \} \rangle$.
\end{itemize}

\smallskip
\begin{lemma}\label{5.2}
Let $R$ be a domain and $X_1,X_2 \in R^{\bullet}$.
\begin{enumerate}
\item  $\mathfrak{a}_1(X_1,X_2)^k = \mathfrak{a}_k(X_1,X_2)$ for all  $k \in \N$.
		
\item $\mathfrak{a}_k(X_1,X_2) \cdot \mathfrak{b}_{\ell}(X_1,X_2) = \mathfrak a_{k+\ell}(X_1,X_2)$ \ for all $k, \ell \in \N$ with $k\geq \ell-1$.
		
\item $ \mathfrak{a}_1(X_1,X_2)\cdot \mathfrak{c}_{2k+1}(X_1,X_2) = \mathfrak a_{2k+2}(X_1,X_2)$ \ for all  $k \in \N$.
		
\item $ \mathfrak{a}_1(X_1,X_2)\cdot \mathfrak{c}_{2k}(X_1,X_2) = \mathfrak a_{2k+1}(X_1,X_2)$ \ for all  $k \in \N$.
\end{enumerate}
\end{lemma}

\begin{proof}
This follows by direct calculations.
\end{proof}

\smallskip
\begin{lemma}\label{5.3}
Let $R$ be a domain such that $\mathcal I (R)$ is a \BF-monoid. Suppose there exist distinct $X_1, X_2 \in R^{\bullet}$ such that $\mathfrak{a}_1(X_1,X_2), \mathfrak{b}_2(X_1,X_2)$, and $\mathfrak{c}_{2i+1}(X_1,X_2)$ are atoms of $\mathcal I(R)$ for all $i\in \mathbb N$. Then $\mathcal I (R)$ has the following properties.
\begin{enumerate}
\item $\mathcal{I}(R)$ is not a transfer Krull monoid.
		
\item $\mathcal{I}(R)$ is not  locally finitely generated.
		
\item $\mathcal U_k (\mathcal I(R)) = \N_{\ge 2}$ for all $k\geq 2$.
\end{enumerate}
\end{lemma}

\begin{proof}
1. By Lemma~\ref{5.2} (items 1 and 2), we obtain that
	\[
	\mathfrak{a}_1(X_1,X_2)\cdot\mathfrak{b}_2(X_1,X_2) = \mathfrak a_3 (X_1, X_2) = \mathfrak{a}_1(X_1,X_2)\cdot\mathfrak{a}_2(X_1,X_2) \,.
	\]
	Therefore,  Lemma~\ref{4.8} implies that $\mathcal I (R)$ is not transfer Krull.
	
	2. By   Lemma~\ref{5.2}.3,  the divisor-closed submonoid $\LK \mathfrak{a}_1(X_1,X_2) \RK \subset \mathcal I (R)$ contains infinitely many atoms, whence $\mathcal{I}(R)$ is not locally finitely generated.
	
	3. By Lemma~\ref{3.1}.3, it suffices to prove that  $\mathcal U_2 (\mathcal I(R)) = \N_{\ge 2}$. Since $\mathfrak{c}_{2i+1}(X_1,X_2)$ is an atom of $\mathcal I(R)$ for all $i\in \mathbb N$,  Lemma~\ref{5.2}.3 implies that
	\[
	\{\nu \in \N ~:~ \nu \equiv 0 \mod 2\} \subset \mathcal U_2(\mathcal I(R)).
	\]
	Furthermore, by  Lemma~\ref{5.2} (items 2 and 3), we observe that
	\[
	\mathfrak{a}_{2i}(X_1,X_2)\cdot \mathfrak{b}_2(X_1,X_2) = \mathfrak{a}_{2i-1}(X_1,X_2)\cdot \mathfrak{a}_{3}(X_1,X_2) = \mathfrak{a}_{2i+2}(X_1,X_2) = \mathfrak{a}_1(X_1,X_2)\cdot \mathfrak{c}_{2i+1}(X_1,X_2)
	\]
	for all $i \in \N$, whence
	\[
	\{\nu \in \N_{\geq 3} ~:~ \nu \equiv 1 \mod 2\} \subset \mathcal U_2(\mathcal I(R)) \,.  \qedhere
	\]
\end{proof}

\smallskip
\begin{lemma}\label{5.4}
Let $R$ be a domain and $\mathfrak{q}$ be a $\mathfrak{p}$-primary ideal of $R$ for some $\mathfrak p \in \spec (R)$. Let $\mathfrak{q} = IJ$, where $I$ and $J$ are ideals of $R$ such that $\mathfrak{q} \subsetneq I, J \subsetneq R$. Then $\Rad I = \Rad J = \mathfrak{p}$.
\end{lemma}

\begin{proof}
	Since $\mathfrak{p} = \Rad \mathfrak{q} = \Rad I \cap \Rad J$, so either $\Rad I = \mathfrak{p}$ or $\Rad J = \mathfrak{p}$, say $\Rad I = \mathfrak{p}$. It remains to prove that $\Rad J = \mathfrak{p}$. Clearly $\mathfrak{p} \subset \Rad J$. Conversely, if $g \in \Rad J$, then there exists some positive integer $s$ such that $g^s\in J$. Since $\mathfrak{q} \subsetneq I$,  there exists an element $f \in I\setminus \mathfrak{q}$. As $fg^s \in IJ = \mathfrak{q}$ and $f \not \in \mathfrak{q}$,  there exists some positive integer $t$ such that $g^{st}\in \mathfrak{q}$ and hence $g \in \Rad \mathfrak{q} = \mathfrak{p}$.
\end{proof}

\smallskip
From now on till the end of the proof of Theorem \ref{5.1}, we fix the following notation.   Let $D$ be a domain with quotient field $K$, and let $R = D[X_1,\ldots, X_n]$ and $S = K[X_1, \ldots, X_n]$ be  polynomial rings in $n \ge 2$ variables $X_1, \ldots, X_n$. They are equipped with the natural $\N$-grading such that $\deg (X_1) = \ldots = \deg(X_n)=1$. We set
\[
R = \bigoplus_{t \ge 0} R_t, \quad \text{and} \quad S = \bigoplus_{t \ge 0} S_t \,,
\]
where $R_t \subset S_t$ are the corresponding $t$-components. Every $f \in S$ can be written uniquely in the form $f = \sum_{i \ge 0} f_i$, where $f_i \in S_i$ for all $i \in \N_0$ and $f_i=0$ for all but finitely many $i \in \N_0$.
We denote by $<$ the  lexicographic order on monomials of $K[X_1,\ldots,X_n]$ with $X_1 > X_2 > \ldots > X_n$. For $f\in K[X_1,\ldots,X_n]$, we denote by $\inn(f)$ the \emph{initial monomial} of $f$ with respect to the order $<$.

The \emph{min-degree} $\mdeg(f)$ of a nonzero polynomial $f \in S$ is the smallest nonnegative integer $d$ such that  $f_d$ is nonzero. We set $\mdeg(0) = +\infty$.  The $\mdeg$ function  satisfies the following two properties for all  $f, g \in S$:
\begin{itemize}
	\item [(i)] $\mdeg(fg) = \mdeg(f) + \mdeg(g)$, and
	
	\item [(ii)] $\mdeg(f+g) \geq \min \,\{ \mdeg(f), \, \mdeg(g) \}$,  with equality if $\mdeg(f) \ne \mdeg(g)$.
\end{itemize}
Next we introduce the minimal degree of an ideal of $R$. Let $I \subset R$ be a nonzero ideal. We define the \emph{min-degree}  $\mdeg(I)$ of $I$ to be the smallest nonnegative integer $d$ such that $I$ contains a polynomial whose min-degree is equal to $d$. We set the min-degree of the zero ideal  equal to $+\infty$.

The next lemma says that  the map
\[
\mdeg~:~ \mathcal I(R) \longrightarrow \N_0 \quad \text{ given by } \quad I \longmapsto \mdeg(I)
\]
is a semigroup homomorphism.

\smallskip
\begin{lemma}\label{5.5}
	For every  $I, J \in \mathcal I (R)$, we have $\mdeg(IJ) = \mdeg(I) + \mdeg(J)$.
\end{lemma}

\begin{proof}
We set $d = \mdeg(I), \, e = \mdeg(J)$, $m = \mdeg(IJ) $, and we choose $f \in I$ with $\mdeg (f) = d$, \, $g \in J$ with $\mdeg (g) = e$, and $h \in IJ$ such that $\mdeg (h) = m$. By the above property (i), we have $m \leq d + e$. On the other hand, we have $h = f_1g_1 + \ldots + f_sg_s$,  where $f_i\in I$ and $g_i\in J$ for all $i \in [1,s]$. Then the above property (ii) implies that
\[
m=\mdeg(h) \, \geq \, \min \{\mdeg(f_i)+\mdeg(g_i) \colon i \in [1,s] \} \geq d+e \,. \qedhere
\]
\end{proof}

\smallskip
Let $I, J  \in \mathcal I ( R)$. We set $I[i] = \{f_i\, : \, f\in I \}$ and note that $I[i]  \subset R_i$ is a $D$-module. For $i, j\in \N_0$,
\[
I[i]\cdot J[j] = \Big\{ \sum_{s=1}^ka_sb_s \, : \, k\geq 1, \, a_s \in I[i], \, b_s \in J[j] \Big\}
\]
is also a $D$-module.  Let $I_K[i] = \Span_K\{\,I[i]\,\}$, $J_K[j] = \Span_K\{\,J[j]\,\}$ and, for $m\in \N_0$, $(IJ)_K[m] = \Span_K\{\,(IJ)[m]\,\}$. Clearly, $I_K[i] \subset S_i$, $J_K[j] \subset S_j$, and $(IJ)_K[m] \subset S_m$, whence $I_K[i], J_K[j], (IJ)_K[m]$ are finite dimensional $K$-vector spaces.

\smallskip
\begin{lemma}\label{5.6}
Let $I, J \in \mathcal I (R)$ with	 $\mdeg(I) = d$, $\mdeg(J) = e$, and $\mdeg(IJ) = m$.
	\begin{enumerate}
		\item $(IJ)[m] = I[d] \cdot J[e]$.
		
		\item $(IJ)_K[m] = I_K[d] \cdot J_K[e]$.
	\end{enumerate}
\end{lemma}

\begin{proof}

1. Let $a\in (IJ)[m]$. Then there exists $h\in IJ$ such that $h_m = a$, say $h = \sum_{i=1}^kf^{(i)}g^{(i)}$, where $f^{(i)} \in I$ and $g^{(i)} \in J$ for all $i\in [1, k]$. By Lemma~\ref{5.5}, we obtain that  $a = h_m = \sum_{i=1}^kf_d^{(i)}g_e^{(i)}$, whence $a \in I[d] \cdot J[e]$. Conversely, if $a \in I[d] \cdot J[e]$, then $a=\sum_{i=1}^k a_ib_i$ where $a_i\in I[d], b_i\in J[e]$ for all $i\in [1, k]$. Thus,  there exist $f^{(i)}\in I, g^{(i)}\in J$ such that $f^{(i)}_d = a_i$ and $g^{(i)}_e = b_i$ for all $i\in [1, k]$. Hence $h =\sum_{i=1}^kf^{(i)}g^{(i)} \in IJ$ with $h_m = a \in (IJ)[m]$.
	
	2. Let $a \in (IJ)_K[m]$. Then, by definition, $a = \sum_{\nu=1}^{r} (\alpha_\nu/\beta_\nu) h_\nu$, where $r \in \N$, $h_\nu \in (IJ)[m]$, $\alpha_\nu, \beta_\nu \in D$ and $\beta_\nu \not = 0$ for every $\nu \in [1,r]$. By 1.\,, we obtain $h_\nu = \sum_{j_\nu=1}^{s_\nu} a^{(\nu)}_{j_\nu}b^{(\nu)}_{j_\nu}$, where $a^{(\nu)}_{j_\nu} \in I[d]$ and $b^{(\nu)}_{j_\nu} \in J[e]$ for every $j_\nu \in [1,s_\nu]$. Hence,
\[
a = \sum_{\nu=1}^{r} \sum_{j_\nu=1}^{s_\nu} (\alpha_\nu/\beta_\nu) a^{(\nu)}_{j_\nu}b^{(\nu)}_{j_\nu}  \in I_K[d] \cdot J_K[e] \,.
\]
Conversely, let $\sum_{\nu=1}^r a_\nu b_\nu \in I_K[d] \cdot J_K[e]$. Then
\[
a_\nu = \sum_{j_\nu=1}^{s_\nu}(\alpha_{\nu, j_\nu}/\beta_{\nu, j_\nu}) a^{(\nu)}_{j_\nu} \quad \text{ and } \quad b_\nu = \sum_{k_\nu=1}^{t_\nu}(\gamma_{\nu, k_\nu}/\theta_{\nu, k_\nu}) b^{(\nu)}_{k_\nu} \,,
\]
where $a^{(\nu)}_{j_\nu} \in I[d]$, $b^{(\nu)}_{k_\nu} \in J[e]$ and $\alpha_{\nu, j_\nu}, \beta_{\nu, j_\nu}, \gamma_{\nu, k_\nu}, \theta_{\nu, k_\nu} \in D$ with $\beta_{\nu, j_\nu} \not = 0 \not = \theta_{\nu, k_\nu}$ for all $\nu, j_\nu, k_\nu$. If we denote $\beta_\nu= \prod_{j_\nu=1}^{s_\nu}\beta_{\nu, j_\nu}$ and $\theta_\nu= \prod_{k_\nu=1}^{t_\nu}\theta_{\nu, k_\nu}$, then $a_\nu b_\nu$ can be written as
	\[
	a_\nu b_\nu = (1/\beta_\nu \theta_\nu) \sum_{j_\nu=1}^{s_\nu}\sum_{k_\nu=1}^{t_\nu}\tau_{j_\nu, k_\nu}a^{(\nu)}_{j_\nu}b^{(\nu)}_{k_\nu} \,,
	\]
	where all $\tau_{j_\nu, k_\nu} \in D$. Setting $\beta = \prod_{\nu=1}^{r}\beta_\nu$ and $\theta = \prod_{\nu=1}^{r}\theta_\nu$ we obtain that
	\[
	\sum_{\nu=1}^r a_\nu b_\nu = (1/\beta\theta)\sum_{\nu=1}^r \sum_{j_\nu=1}^{s_\nu}\sum_{k_\nu=1}^{t_\nu}\rho_{j_\nu, k_\nu}a^{(\nu)}_{j_\nu}b^{(\nu)}_{k_\nu},
	\]
	where all $\rho_{j_\nu, k_\nu} \in D$. Therefore, $\beta\theta(\sum_{\nu=1}^r a_\nu b_\nu) \in I[d] \cdot J[e]$ and hence, again by using 1.\,, we get
	\[
	\beta\theta(\sum_{\nu=1}^r a_\nu b_\nu) = h_m,
	\]
	where $h \in IJ$. Thus, we obtain that $\sum_{\nu=1}^r a_\nu b_\nu \in (IJ)_K[m]$.
\end{proof}

Before moving further we demonstrate in  simple special cases how our techniques work for studying factorizations in the monoid of nonzero ideals of polynomial rings.

\smallskip
\begin{example}\label{5.7}~

1.  We claim that the ideal $\langle X^2, Y^2 \rangle \subset K[X,Y]$ is an atom of $\mathcal{I}(K[X,Y])$. Assume to the contrary that there are two nonzero proper ideals $I, J$ of $K[X,Y]$ with $\mdeg(I) = d$, $\mdeg(J) = e$ and
\[
\langle X^2, Y^2 \rangle = IJ.
\]
By Lemma~\ref{5.6} we obtain
\begin{equation}\label{ex1}
	\Span_K\{X^2, Y^2\} = I_K[d] \cdot J_K[e],
\end{equation}
and we have  $d+e = 2$ by Lemma~\ref{5.5}. Furthermore,  we have $\Rad I = \Rad J = \langle X, Y \rangle$, whence $d = e = 1$. Clearly $\dim_KI[d]\geq 1$ and $\dim_K J[e]\geq 1$. Assume to the contrary that one of these dimensions equals one, say   $I_K[d] = \Span_K\{f\}$ and $J_K[e] = \Span_K\{g_1,\ldots, g_s\}$. Then $X^2 = \sum_{j=1}^{s}\alpha_jfg_j$ where $\alpha_j \in K$ for all $j\in [1,s]$. This implies that $f = \alpha X$ for some $\alpha \in K^\times$, which is not possible since $Y^2 \in I_K[d]\cdot J_K[e]$. Therefore,  $\dim_K I_K[d] \ge 2$ and $\dim_KJ[e] \ge 2$. Since $I_K[d], J_K[e] \subset \Span_K\{X,Y\}$, it follows that  $I_K[d] = J_K[e] = \Span_K\{X,Y\}$. This implies that $I_K[d] \cdot J_K[e] = \Span_K\{X^2,XY, Y^2\}$, a contradiction to (\ref{ex1}).
		
\smallskip
2. We claim that  the ideal $\langle X^3+Y^3, X^2Y, XY^2 \rangle \subset K[X,Y]$ is an atom of $\mathcal{I}(K[X,Y])$. Assume to the contrary that there are   two nonzero proper ideals $I, J$ of $K[X,Y]$ with $\mdeg(I) = d$, $\mdeg(J) = e$ and
		\[
		\langle X^3+Y^3, X^2Y, XY^2 \rangle = IJ.
		\]
By Lemma~\ref{5.6} we obtain
\begin{equation}\label{ex2}
	\Span_K\{X^3+Y^3, X^2Y, XY^2\} = I_K[d] \cdot J_K[e],
\end{equation}
and we have	 $d+e = 3$ by Lemma~\ref{5.5}. Since $\langle X, Y \rangle^4 \subset \langle X^3+Y^3, X^2Y, XY^2 \rangle$, we infer that $\Rad I = \Rad J = \langle X, Y \rangle$. Thus,  $d, e \geq 1$, and after renumbering if necessary we suppose that  $d=1$ and $e=2$. Clearly $\dim_KI[d]\geq 1$ and $\dim_KJ[e]\geq 1$.
Assume to the contrary that one of these dimensions equals one, say
$I_K[d] = \Span_K\{f\}$ and $J_K[e] = \Span_K\{g_1,\ldots, g_s\}$. Then $X^2Y = \sum_{j=1}^{s}\alpha_jfg_j$ where $\alpha_j \in K$ for all $j\in [1,s]$. This implies that $f = \alpha X$ or $f=\alpha' Y$ for some $\alpha, \alpha' \in K^\times$, which is not possible since $X^3+Y^3 \in I_K[d]\cdot J_K[e]$. Therefore,  $\dim_K I_K[d] \ge 2$ and $\dim_KJ[e] \ge 2$. Since $I_K[d] \subset \Span_K\{X,Y\}$, it follows that $I_K[d] = \Span_K\{X,Y\}$. For every $i \in [1,s]$, we have $Xg_i \in I_K[d] \cdot J_K[e]$ and thus,  by (\ref{ex2}),
		\[
		Xg_i = \beta_{i,1}(X^3+Y^3) + \beta_{i,2} X^2Y + \beta_{i,3} XY^2
		\]
		where $\beta_{i,j} \in K$ for every $j\in [1,3]$. This implies that $\beta_{i,1} = 0$ and $g_i = \beta_{i,2}XY+\beta_{i,3} Y^2$. Similarly, $Yg_i \in I_K[d] \cdot J_K[e]$ yields
		\[
		Yg_i = \gamma_{i,1}(X^3+Y^3) + \gamma_{i,2}X^2Y + \gamma_{i,3}XY^2
		\]
		where $\gamma_{i,j}\in K$ for every $j\in [1,3]$. This implies that $\gamma_{i,1} = 0$ and $g_i = \gamma_{i,2} X^2+\gamma_{i,3} XY$. Hence $\beta_{i,3}=\gamma_{i,2}=0$ and $g_i = \beta_{i,2}XY$. Therefore,  $J_K[e] = \Span_K\{XY\}$, a contradiction to $\dim_K J_K[e]\geq 2$.
		
3. The following equation
		\[
		\langle X^2, Y^2 \rangle \langle X^3+Y^3, X^2Y, XY^2 \rangle = \langle X, Y \rangle^5
		\]
involves only atoms of $\mathcal{I}(K[X,Y])$, whence it shows that $2, 5 \in \mathsf L_{\mathcal I (K[X,Y]) } ( \langle X, Y \rangle^5 )$.
\end{example}

\smallskip
Example~\ref{5.7}.1 was already settled in \cite[Proposition~4.6]{He-La-Re21}, but our techniques allow us to study  polynomial ideals over   domains whose semigroup of nonzero ideals is a BF-monoid (see Proposition \ref{5.10}).
For $m \in \N_0$, let $\mathcal M_{m;1,2}$ denote the set of all monomials of the form $X_1^rX_2^s$ with $r+s=m$ and $r,s \in \N_0$.

\smallskip
\begin{lemma}\label{5.8}
The  ideals
\[
\mathfrak c'(X_1,X_2) = \langle X_1^3+X_2^3, X_1^2X_2, X_1X_2^2 \rangle \quad \text{and} \quad
		\mathfrak a(X_1,X_2) = \langle \{X_1^m, X_2^m\} \cup \mathcal N  \rangle \,,
\]
where $m \in \N$ and  $\mathcal N \subset \mathcal M_{m;1,2}$ is any subset, are $\langle X_1, X_2 \rangle$-primary in $R$.
In particular, if $\mathfrak{c}'(X_1,X_2) = IJ$ or $\mathfrak{a}(X_1,X_2) = IJ$, where $I$ and $J$ are ideals of $R$ such that $\mathfrak{c}'(X_1,X_2) \subsetneq I, J \subsetneq R$ or $\mathfrak{a} (X_1,X_2) \subsetneq I, J \subsetneq R$, then $\mdeg (I) \ge 1$ and $\mdeg(J) \ge 1$.
\end{lemma}

\begin{proof}
We set $\mathfrak{c}' := \mathfrak{c}'(X_1,X_2)$ and $\mathfrak{a} := \mathfrak{a}(X_1,X_2)$. It is a well known fact that if an ideal $\mathfrak{q} \subset R$ is $\mathfrak{p}$-primary, then its extension $\mathfrak{q}[X] \subset R[X]$ is $\mathfrak{p}[X]$-primary, cf. \cite[Exercise 4.7 (iii)]{At-Ma69}. We proceed  by induction on the number of indeterminates $n \ge 2$.
	
	Let $n=2$. Then the ideal extensions $\mathfrak c'K[X_1,X_2]$ and $\mathfrak aK[X_1,X_2]$ are $\langle X_1, X_2 \rangle K[X_1,X_2]$-primary, because $\langle X_1, X_2 \rangle K[X_1,X_2]^4 \subset \mathfrak c'K[X_1,X_2]$, $\langle X_1, X_2 \rangle K[X_1,X_2]^{2m-1} \subset \mathfrak a K[X_1,X_2]$ and $\langle X_1, X_2 \rangle K[X_1,X_2]$ is a maximal ideal. Therefore the ideal contractions $\mathfrak c'K[X_1,X_2]\cap D[X_1,X_2]$ and $\mathfrak aK[X_1,X_2]\cap D[X_1,X_2]$ are $\langle X_1, X_2 \rangle K[X_1,X_2]\cap D[X_1,X_2]$-primary. If we prove that
	\[
	\langle X_1, X_2 \rangle K[X_1,X_2]\cap D[X_1,X_2] \subset \langle X_1, X_2 \rangle, \mathfrak c'K[X_1,X_2]\cap D[X_1,X_2] \subset \mathfrak c' \text{ and } \mathfrak aK[X_1,X_2]\cap D[X_1,X_2] \subset \mathfrak a,
	\]
	then we are done. Let $f/\alpha = g \in \langle X_1, X_2 \rangle K[X_1,X_2]\cap D[X_1,X_2]$, where $f\in \langle X_1, X_2 \rangle, g\in D[X_1,X_2]$ and $\alpha \in D^{\bullet}$. Then $\alpha g\in \langle X_1, X_2 \rangle$ and hence $g\in \langle X_1, X_2 \rangle$.
	
	Let now $f/\alpha = g \in \mathfrak c'K[X_1,X_2]\cap D[X_1,X_2]$, where $f\in \mathfrak c', g\in D[X_1,X_2]$ and $\alpha \in D^{\bullet}$. Then $\alpha g\in \mathfrak c'$ and it only requires to show that $g\in \mathfrak c'$. But $\alpha g\in \mathfrak c'$ implies that $g = \beta(X_1^3+X_2^3)+g_{21}X_1^2X_2+g_{12}X_1X_2^2 + g'$, where $\beta,g_{21},g_{12}\in D$ and $g'\in D[X_1,X_2]$ with $\mdeg(g') = 4$. Since $g'\in \langle X_1, X_2 \rangle^4 \subset \mathfrak c'$, so $g\in \mathfrak c'$.
	
	Let now $f/\alpha = g \in \mathfrak aK[X_1,X_2]\cap D[X_1,X_2]$, where $f\in \mathfrak a, g\in D[X_1,X_2]$ and $\alpha \in D^{\bullet}$. Then $\alpha g\in \mathfrak a$ and hence by using similar calculations as above we get $g\in \mathfrak a$.
	
	Assume now the result is true for $n\geq 2$ and consider the ring $R[X_{n+1}]$. So $\mathfrak{c}'[X_{n+1}]$ and $\mathfrak{a}[X_{n+1}]$ are $\langle X_1, X_2 \rangle [X_{n+1}]$-primary ideals of $R[X_{n+1}]$, but in $R[X_{n+1}]$ we have
	\[
	\mathfrak{c}'[X_{n+1}] = \mathfrak{c}', \quad \mathfrak{a}[X_{n+1}] = \mathfrak{a} \quad \text{ and } \quad \langle X_1, X_2 \rangle [X_{n+1}] = \langle X_1, X_2 \rangle.
	\]
	Thus $\mathfrak a$ and $\mathfrak c'$ are $\langle X_1, X_2 \rangle$-primary. The in particular statement  follows by Lemma~\ref{5.4}.
\end{proof}

\smallskip
\begin{lemma}\label{5.9}
	Let $V \subset S_d$, with $d \in \N$,  be a $K$-vector subspace of dimension $\dim_K (V) = s$. Then there are linearly independent elements $f_1,\ldots, f_s \in V$ such that $\inn(f_1) > \ldots > \inn(f_s)$.
\end{lemma}

\begin{proof}
	Let $g_1,\ldots, g_s$ be linearly independent elements of $V$. If we consider the canonical isomorphism $S_d \cong K^{\binom{d+n-1}{d}}$, then $g_1,\ldots, g_s$ will represent the corresponding linearly independent vectors of $V \subset K^{\binom{d+n-1}{d}}$. Let $A$ be an $s\times \binom{d+n-1}{d}$ matrix whose rows are $g_1,\ldots, g_s$. Now we apply the Gaussian elimination on $A$ and reduce it to row echelon form. The resulting matrix is represented by rows, say $f_1,\ldots, f_s \in V$, which are again linearly independent elements and satisfy our requirement $\inn(f_1) > \ldots > \inn(f_s)$.
\end{proof}

\smallskip
\begin{proposition} \label{5.10}
Suppose that $\mathcal I (R)$ is a \BF-monoid. Then the following ideals are atoms of $\mathcal I (R)$.
	\begin{enumerate}
		\item $\mathfrak{b}_i(X_1,X_2)$ for every $i\in \N$,
		
		\item $\mathfrak{c}_{2i+1}(X_1,X_2)$ for every $i\in \N$,
		
		\item $\mathfrak{c}_{2i}(X_1,X_2)$ for every $i\in \N_{\geq 3}$, and
		
		\item $\mathfrak c'(X_1,X_2) = \langle X_1^3+X_2^3, X_1^2X_2, X_1X_2^2 \rangle$.
	\end{enumerate}
\end{proposition}

\begin{proof}
For all $i \in \N$, we use the abbreviations $\mathfrak{b}_i := \mathfrak{b}_i(X_1,X_2), \mathfrak{c}_{2i+1} := \mathfrak{c}_{2i+1}(X_1,X_2), \mathfrak{c}_{2i}:= \mathfrak{c}_{2i}(X_1,X_2)$, and $\mathfrak{c}' := \mathfrak c'(X_1,X_2)$. In order to show that an ideal $\mathfrak a$ is an atom of $\mathcal I (R)$, it suffices to show that there are no ideals $I, J \in \mathcal I (R)$ with $\mathfrak a \subsetneq I, J \subsetneq R$ such that $\mathfrak  a = I J$, because $\mathcal I (R)$ is a reduced unit-cancellative semigroup.
	
	1. Let $i\in \N$ and assume to the contrary that $\mathfrak{b}_i = IJ$ with $\mathfrak{b}_i \subsetneq I, J \subsetneq R$ such that $\mdeg(I) = d$ and $\mdeg(J) = e$. Then, by Lemma~\ref{5.8}, we have $d\geq 1$ and $e \geq 1$, and by Lemma~\ref{5.6}.2 we obtain that
	\begin{equation}\label{atom1}
	\Span_K\{X_1^i, X_2^i\} = I_K[d] \cdot J_K[e].
	\end{equation}
	Note that $I_K[d]$ and $J_K[e]$ are finite vector spaces and by Equation~\ref{atom1} it is not possible that $\dim_K I_K[d] = \dim_K J_K[e] = 1$. Thus, without loss of generality, we may assume that $\dim_K I_K[d] \geq 2$ and $\dim_K J_K[e] \geq 1$. Let $f_1, f_2 \in I_K[d]$ be linearly independent and $g \in J_K[e]$ any nonzero element.
	
By Lemma~\ref{5.9}, we may assume that $\inn(f_1) > \inn(f_2)$. If $\inn(f_i)$ does not equal $X_1^d$ or $X_2^d$, then $\inn(f_ig) \, (=\inn(f_i)\inn(g))$ does not equal $X_1^i$ or $X_2^i$, which is not possible by (\ref{atom1}). This means we must have $\inn(f_1) = X_1^d$ and $\inn(f_2) = X_2^d$. If $\inn(g) = X_1^e$, then $\inn(f_2g) = X_1^eX_2^d$, a contradiction to (\ref{atom1}). If  $\inn(g) \ne  X_1^e$, say $\inn(g) = X_1^{a_1} \cdot \ldots \cdot X_n^{a_n}$ such that $\sum_{j=1}^n a_i=e$ with $a_1<e$ and $a_j\geq 1$ for some $j\in [2, n]$, then $\inn(f_1g)$ is divisible by $X_1X_j$, again a contradiction to (\ref{atom1}).
	
\smallskip
	2. Let $i\in \N$ and assume to the contrary that $\mathfrak{c}_{2i+1} = IJ$ with $\mathfrak{c}_{2i+1} \subsetneq I, J \subsetneq R$ such that $\mdeg(I) = d$ and $\mdeg(J) = e$. Then, by Lemma~\ref{5.8}, we have $d\geq 1$ and $e \geq 1$, by Lemma~\ref{5.5}, we have  $2i+1= d+e$, and  Lemma~\ref{5.6}.2 implies that
	\begin{equation}\label{atom2}
	\Span_K\{\{X_1^{2i+1},X_1^{2i}X_2\}\cup \{X_1^{2i-j}X_2^{j+1}~:~ j\in[1,2i+1] \text{ and } j \equiv 0 \mod 2 \}\} = I_K[d] \cdot J_K[e].
	\end{equation}
	We claim that $\dim_KI[d] \ge 2$ and $\dim_KJ[e] \ge 2$. Indeed, if $I_K[d] = \Span_K\{f\}$ and $J_K[e] = \Span_K\{g_1,\ldots, g_s\}$, then $I_K[d] \cdot J_K[e] = \Span_K\{fg_1,\ldots, fg_s\}$. From (\ref{atom2}), we get
	\[
	X_1^{2i+1} = \sum_{j=1}^{s} \alpha_j fg_j, \text{ where } \alpha_j\in K \text{ for all } j\in [1,s] \,,
	\]
	and  we deduce $f =  \alpha X_1^d$ for some $\alpha \in K^\times$. Hence $X_2^{2i+1}$ cannot belong to $I_K[d] \cdot J_K[e]$, a contradiction to (\ref{atom2}).
	
	Let $\dim_K I[d] = r\geq 2$ and $f_1, f_2, \ldots, f_r \in I_K[d]$ be linearly independent such that $\inn(f_1) > \inn(f_2) > \ldots > \inn(f_r)$ (we use Lemma~\ref{5.9}). Similarly, let $\dim_K J_K[e] = s \geq 2$ and let $g_1, \ldots, g_s\in J_K[e]$ be linearly independent such that $\inn(g_1)  > \ldots > \inn(g_s)$. If $\inn(f_1) \ne X_1^d$, then $\inn (f_1g) \ne X_1^{2i+1}$ for any $g \in J_K[e]$, which is not possible by \ref{atom2}. Thus, we have  $\inn(f_1) = X_1^d$, and hence $\inn(g_1) = X_1^e$. Without loss of generality assume that $d$ is odd, while $e$ is even.  If  $\inn(f_2) = X_1^{d-u}X_2^u$ for an even $u \in \N$, then $\inn(f_2g_1) = X_1^{d+e-u}X_2^u$, a contradiction to (\ref{atom2}). If $\inn(f_2) = X_1^{d-u}X_2^u$ for an odd $u \in \N$, then we consider $\inn(g_2)$. If  $\inn(g_2) = X_1^{e-v}X_2^v$ for an odd $v \in \N$, we get a contradiction as $\inn(f_2g_2) = X_1^{d+e-u-v}X_2^{u+v}$. If  $\inn(g_2) = X_1^{e-v}X_2^v$ for an even $v \in \N$, we get a contradiction as $\inn(f_1g_2) = X_1^{d+e-v}X_2^{v}$. Note that a case of $\inn(f_{i_0}) = X_1^{j_1} \cdot \ldots \cdot X_n^{j_n}$ (similarly for $g_{i_0}$) with $\sum_{\nu =1}^{n}j_{\nu}=d$ and $j_{\nu}>0$ for some $\nu \in [3,n]$ is not possible by the same argument.

\smallskip	
	3. Let $i\in \N_{\ge 3}$ and assume to the contrary that $\mathfrak{c}_{2i} = IJ$ with $\mathfrak{c}_{2i} \subsetneq I, J \subsetneq R$ such that $\mdeg(I) = d$ and $\mdeg(J) = e$. Then, by Lemma~\ref{5.8}, we have $d\geq 1$ and $e \geq 1$, by Lemma~\ref{5.5}, we have  $2i= d+e$, and Lemma~\ref{5.6}.2 implies that
	\begin{equation}\label{atom3}
	\Span_K\{\{X_1^{2i}, X_1^{2i-1}X_2\}\cup \{X_1^{2i-j}X_2^{j}~:~ j\in[1,2i] \text{ and } j \equiv 0 \mod 2 \}\} = I_K[d]\cdot J_K[e].
	\end{equation}
	As in 2., $I_K[d]$ and $J_K[e]$ are finite dimensional vector spaces of dimension at least two. Moreover, since $i\geq 3$,  $\dim_K I_K[d]$ or $\dim_K J_K[e]$ must be at least three, say $\dim_K I[d] \geq 3$. Let $\dim_K I[d] = r\geq 3$ and let $f_1,  \ldots, f_r \in I_K[d]$ be linearly independent such that $\inn(f_1) > \inn(f_2) > \ldots > \inn(f_r)$ (again we use Lemma~\ref{5.9}). Similarly, let $\dim_K J_K[e] = s \geq 2$ and let $g_1, \ldots, g_s\in J_K[e]$ be linearly independent such that $\inn(g_1) >  \ldots > \inn(g_s)$. As in 2., we have $\inn(f_1) = X_1^d$, and hence $\inn(g_1) = X_1^e$.
	
	Consider now $f_2$ and $g_2$ such that $\inn(f_2) = X_1^{d-a}X_2^a$ and $\inn(g_2) = X_1^{e-b}X_2^b$, where $a, b \in \N$. We distinguish two cases.

\noindent
CASE 1: $a = 1$.

 If $b$ is even, then $\inn(f_2g_2) = X_1^{d+e-b-1}X_2^{b+1}$ which is not possible, see (\ref{atom3}). If $b \geq 3$ is odd, then $\inn(f_1g_2) = X_1^{d+e-b}X_2^{b}$ which again is not possible. Thus,  $b=1$ and $\dim_K J_K[e] = 2$. Consider now $f_3$ such that $\inn(f_3)=X_1^{d-c}X_2^c$ with $c\in \N_{\geq 2}$. If $c$ is even, then  $\inn(f_3g_2) = X_1^{d+e-c-1}X_2^{c+1}$, a contradiction to (\ref{atom3}), and if $c$ is odd, then $\inn(f_3g_1) = X_1^{d+e-c}X_2^{c}$, again a contradiction to (\ref{atom3}). Hence $\dim_K I_K[d] = 2$, which is not possible.

\noindent
CASE 2: $a \geq 2$.

 If $a$ is odd, then $\inn(f_2g_1) = X_1^{d+e-a}X_2^{a}$, which is a contradiction to (\ref{atom3}). Therefore,  $a$ is  even and, hence $b$ is also even. In general, there are only the following possibilities for $\inn(f_u)$ and $\inn(g_v)$:
	\[
	\inn(f_u) = X_1^{d-a_u}X_2^{a_u} \text{ with all } a_u \text{  even and } 0= a_1 < \ldots < a_r \leq d
	\]
	and
	\[
	\inn(g_v) = X_1^{e-b_v}X_2^{b_v} \text{ with all } b_v \text{  even and } 0= b_1 < \ldots < b_s \leq e.
	\]
	Now as $X_1^{2i-1}X_2 \in I_K[d]\cdot J_K[e]$, so $X_1^{2i-1}X_2 = \sum_{u=1}^{r}\sum_{v=1}^{s}\alpha_{uv}f_ug_v$ with $\alpha_{uv}\in K$ for every $u,v$. If $\alpha_{11}\not=0$, then initial monomial of the right hand side would be $X_1^{2i}$, a contradiction. If $\alpha_{11}=0$, then
	\[
	X_1^{2i-1}X_2 = \inn \left(\sum_{u=1}^{r}\sum_{v=1}^{s}\alpha_{uv}f_ug_v \right) \leq \max \big\{\inn(f_ug_v) \colon u \in [1,r], v \in [1,s] \big\}
	\]
	which is not possible, since the right hand side of the above inequality is $X_1^{d+e-a_{u_0}-b_{v_0}}X_2^{a_{u_0}+b_{v_0}}$ which is always less than $X_1^{2i-1}X_2$.

\smallskip	
	4. Assume to the contrary that $\mathfrak{c}' = IJ$ with $\mathfrak{c}' \subsetneq I, J \subsetneq R$ such that $\mdeg(I) = d$ and $\mdeg(J) = e$. Then, by Lemma~\ref{5.8}, we have $d\geq 1$ and $e \geq 1$, by Lemma~\ref{5.5}, we have $d+e=3$,  and  Lemma~\ref{5.6}.2 implies that
	\begin{equation}\label{atom4}
	\Span_K\{X_1^3+X_2^3, X_1^2X_2, X_1X_2^2\} = I_K[d]\cdot J_K[e].
	\end{equation}
	We may assume that $d=1$ and $e=2$. We claim that $\dim_KI[d] \ge 2$ and $\dim_KJ[e] \ge 2$. Indeed, if $I_K[d] = \Span_K\{f\}$ and $J_K[e] = \Span_K\{g_1,\ldots, g_s\}$, then $I_K[d] \cdot J_K[e] = \Span_K\{fg_1,\ldots, fg_s\}$. From (\ref{atom4}), we get
	\[
	X_1^{2}X_2 = \sum_{j=1}^{s} \alpha_j fg_j, \text{ where } \alpha_j\in K \text{ for all } j\in [1,s] \,,
	\]
	and  we deduce either $f =  \alpha X_1$ or $f =  \alpha' X_2$ for some $\alpha, \alpha' \in K^\times$. Hence $X_1^3+X_2^{3}$ cannot belong to $I_K[d] \cdot J_K[e]$, a contradiction to (\ref{atom4}). Note that similar argument works if we consider $d=2$.
	
	Let $\dim_K I[d] = r\geq 2$ and let $f_1,  \ldots, f_r \in I_K[d]$ be linearly independent such that $\inn(f_1)  > \ldots > \inn(f_r)$ (we use Lemma~\ref{5.9}). Similarly, let $\dim_K J_K[e] = s \geq 2$ and let $g_1, \ldots, g_s\in J_K[e]$ be linearly independent such that $\inn(g_1) >  \ldots > \inn(g_s)$. On the other hand, for $j\in [3,n]$ $\inn(f_{i_0})=X_j$ is not possible, whence $\dim_K I[d] = 2$ and $\inn(f_{1})=X_1, \inn(f_{2})=X_2$. Similarly, we obtain that $\dim_K J[e] = 2$ and $\inn(g_{1})=X_1^2, \inn(g_{2})=X_1X_2$. Assume $g_1 = X_1^2+aX_2^2+\cdots$\,. Then by (\ref{atom4})
	\[
	f_2g_1 = \alpha_1(X_1^3+X_2^3) + \alpha_2X_1^2X_2 + \alpha_3X_1X_2^2, \qquad \text{ where } \alpha_1, \alpha_2, \alpha_3\in K
	\]
	which implies that $a = 0$. Similarly, the coefficient of the term $X_2^2$ in $g_2$ is zero as well. This shows that $X_1^3+X_2^3$ does not belong to $I_K[d] \cdot J_K[e]$, a contradiction.
\end{proof}

\smallskip
\begin{remark}\label{5.11}
	If $2$ is a unit of $D$, then
	\[
	\mathfrak{c}_4(X_1,X_2) = \langle X_1^4, X_1^3X_2, X_1^2X_2^2, X_2^4 \rangle = \langle X_1^2, X_1X_2+X_2^2 \rangle\langle X_1^2, X_1X_2-X_2^2 \rangle \,,
	\]
whence $\mathfrak{c}_4(X_1,X_2)$ is not an atom of $\mathcal I(R)$.
\end{remark}

\medskip
\begin{proof}[Proof of Theorem \ref{5.1}]
1. and 2.
We need to show the claim concerning the finite factorization property. All other statements follow by Lemma~\ref{5.3} and Proposition~\ref{5.10} (clearly, $\mathfrak a_1 (X_1, X_2)$ is an atom). Suppose that  $D^{\times}$ is infinite. For every $\alpha \in D^{\times}$, we have the identity
\[
\langle X_1, X_2\rangle \cdot \langle X_1^2+\alpha X_2^2, X_1X_2\rangle = \langle X_1, X_2\rangle^3 \,.
\]
If $\alpha, \alpha' \in D^{\times}$ are distinct, then $\langle X_1^2+\alpha X_2^2, X_1X_2\rangle \ne \langle X_1^2+\alpha' X_2^2, X_1X_2\rangle$. Thus, the element $\langle X_1, X_2\rangle^3$ has infinitely many divisors, whence $\mathcal I (R)$ is not an \FF-monoid by \cite[Proposition 1.5.5]{Ge-HK06a}.

3.  In order to show that $\mathsf L_{\mathcal{I}(R)} \big(\mathfrak{a}_k(X_1,X_2) \big) =  \mathsf L \big(\mathfrak{a}_k(X_1,X_2) \big)\subset [2,k]$ for all $k \ge 2$, we choose $\ell \in \N_{\ge 2}$ and consider a factorization $ \mathfrak{a}_k(X_1,X_2) = I_1 \cdot \ldots \cdot I_{\ell}$, where $I_1, \ldots, I_{\ell}$ are atoms in $\mathcal I (R)$. By Lemma \ref{5.8}, we have $\mdeg (I_j) \ge 1$ for all $j \in [1, \ell]$. Using Lemma \ref{5.5} we obtain that
\[
k = \mdeg ( \mathfrak a_k) = \sum_{j=1}^{\ell} \mdeg (I_j) \ge \ell \,,
\]
whence $\ell \le k$ and $\mathsf L \big(\mathfrak{a}_k(X_1,X_2) \big) \subset [2,k]$.
To verify the reverse inclusion  we proceed by induction on $k$. This is  clear for $k=2$. By Lemma~\ref{5.2}.2 (with $\ell=2$ and $k=1$), we obtain that
\[
\mathfrak{a}_3(X_1,X_2) = \mathfrak{a}_1(X_1,X_2)\cdot \mathfrak{b}_2(X_1,X_2) \,.
\]
Since the involved ideals are atoms by Proposition~\ref{5.10}.1, the assertion holds for $k=3$.
Suppose that claim holds for $k\geq 3$. Since $\mathfrak{a}_{k+1}(X_1,X_2) = \mathfrak{a}_1(X_1,X_2)\cdot \mathfrak{a}_k(X_1,X_2)$ and $\mathfrak{a}_1(X_1,X_2)$ is an atom, it follows that
\[
[3, k+1] \subset 1 + \mathsf L \big( \mathfrak{a}_k(X_1,X_2) \big) \subset \mathsf L \big( \mathfrak{a}_{k+1} (X_1, X_2) \big) \,.
\]
It remains to verify that $2 \in \mathsf L \big(\mathfrak{a}_{k+1}(X_1,X_2) \big)$. If $k\in \N_{\ge 3} \setminus \{4\}$, then the following identity (see  Lemma~\ref{5.2})
\[
\mathfrak{a}_{k+1}(X_1,X_2) =  \mathfrak{a}_1(X_1,X_2)\cdot \mathfrak{c}_k(X_1,X_2)
\]
together with Proposition~\ref{5.10} show that $2 \in \mathsf L \big(\mathfrak{a}_{k+1}(X_1,X_2) \big)$. If $k = 4$, then $2 \in \mathsf L \big(\mathfrak{a}_{k+1}(X_1,X_2) \big)$ because  $\mathfrak{a}_5(X_1,X_2)=\mathfrak {b}_2(X_1,X_2)\cdot\mathfrak{c}'(X_1,X_2)$ (use Proposition~\ref{5.10}.4).

4. We set $\mathfrak{p}_0 = \langle X_1 \rangle$, $H_1 = \mathcal{F}(\{\mathfrak{p}_0\})$ and $H_2 = \{\mathfrak{a} \in \mathcal{I}(R)~:~ \mathfrak{p}_0 \nmid \mathfrak{a}\}$. Since $\mathfrak{p}_0$ is an invertible prime ideal, it is a cancellative prime element of $\mathcal I (R)$, whence we get that
\[
\mathcal I(R) = H_1\times H_2 \,.
\]
Note that  $\mathfrak{a}_k(X_1,X_2)\in H_2$ for all $k\in \N$. Thus, 3. shows that Conditions (a) and (b') of Proposition~\ref{3.2} hold, whence $\mathcal I (R)$ is fully elastic.
\end{proof}

\smallskip
Theorem \ref{5.1} shows, among others, that unions of sets of lengths of the monoid of all nonzero ideals are equal to $\N_{\ge 2}$, as it is true for transfer Krull monoids (which include monoids of invertible ideals of Krull domains) with  infinite class group and prime divisors in all classes (Proposition \ref{4.9}). We post the conjecture that also their sets of lengths coincide, namely that every finite subset $L \subset \N_{\ge 2}$ occurs as a set of lengths.

\smallskip
\begin{conjecture} \label{5.12}
Let  $R = D[X_1, \ldots, X_n]$ be the polynomial ring  in $n \ge 2$ indeterminates over a domain $D$, and suppose that $\mathcal I (R)$ is a \BF-monoid.  Then,  for every finite subset $L \subset \N_{\ge 2}$, there is $\mathfrak a \in \mathcal I(R)$ such that $\mathsf L_{\mathcal I(R)} (\mathfrak a) = L$.
\end{conjecture}

To conclude this paper, we would like to compare the arithmetic of $\mathcal I (R)$, in particular Theorem \ref{5.1} and Conjecture \ref{5.12}, with the arithmetic of the power monoid of $\N_0$. Following the terminology and notation of Fan and Tringali \cite{Fa-Tr18a}, we denote by
\begin{itemize}
\item $\mathcal P_{\fin} (\N_0)$ the {\it power monoid} of $\N_0$, that is the semigroup of finite nonempty subsets of $\N_0$ with set addition as operation (i.e., for finite nonempty subsets $A, B \subset \N_0$, their sumset $A+B$ is defined as $A+B = \{ a + b \colon a \in A, b \in B\}$), and by

\item  $\mathcal P_{\fin, 0} (\N_0)$ the {\it reduced power monoid} of $\N_0$, that is the subsemigroup of $\mathcal P_{\fin} (\N_0)$ consisting of all finite nonempty subsets of $\N_0$ that contain $0$.
\end{itemize}
Both, $\mathcal P_{\fin} (\N_0)$ and $\mathcal P_{\fin, 0} (\N_0)$, are commutative reduced unit-cancellative semigroups (whence  monoids in the present sense) and $\{0\}$ is their zero-element. Power monoids are objects of primary interest in additive combinatorics and their arithmetic is studied in detail by Antoniou, Fan,  and Tringali in \cite{Fa-Tr18a, An-Tr21a}. Among others, they show that $\mathcal P_{\fin} (\N_0)$ is not transfer Krull,  that unions of sets of lengths of $\mathcal P_{\fin} (\N_0)$ are equal to $\N_{\ge 2}$, and that the set of distances equals $\N$. The standing conjecture is that every finite subset $L \subset \N_{\ge 2}$ occurs as a set of lengths of $\mathcal P_{\fin} (\N_0)$ (\cite[Section 5]{Fa-Tr18a}). Thus, the arithmetic of  $\mathcal I (R)$ and the arithmetic of  $\mathcal P_{\fin} (\N_0)$ seem to have pretty much in common. Our final result shows that the method, developed to show that $\mathcal I (R)$ is fully elastic, also allows to show that $\mathcal P_{\fin} (\N_0)$ is fully elastic, a question that remained open in \cite{Fa-Tr18a}. Moreover, $\mathcal I (R)$ has a submonoid that is isomorphic to $\mathcal P_{\fin} (\N_0)$.

\smallskip
\begin{proposition} \label{5.13}~

\begin{enumerate}
\item The element $\{1\}$ is a cancellative prime element of $\mathcal P_{\fin} (\N_0)$, whence $\mathcal P_{\fin} (\N_0) = F \times \mathcal P_{\fin, 0} (\N_0)$, where $F$ is the free abelian monoid generated by the prime element $\{1\}$. Moreover, $\mathcal P_{\fin} (\N_0)$ is fully elastic.

\item $\mathcal P_{\fin}(\N_0)$ is isomorphic to a submonoid of $\mathcal I (R)$, where $\mathcal I (R)$ is as in Conjecture \ref{5.12}.
\end{enumerate}
\end{proposition}

\begin{proof}
1. It is straightforward to verify that $\{1\}$ is a cancellative prime element of $\mathcal P_{\fin} (\N_0)$. Since $\mathcal P_{\fin, 0} (\N_0) = \{A \in P_{\fin} (\N_0) \colon \{1\} \ \text{does not divide} \ A \}$, it follows that $\mathcal P_{\fin} (\N_0) = F \times \mathcal P_{\fin, 0} (\N_0)$. Proposition 4.8 in \cite{Fa-Tr18a} shows that, for every $n \ge 2$,
\[
[2,n] = \mathsf L_{\mathcal P_{\fin, 0}(\N_0)} ( [0,n]) = \mathsf L_{\mathcal P_{\fin}(\N_0)} ( [0,n]) \,.
\]
Thus, Condition (a) and Condition (b') of Proposition \ref{3.2} are satisfied, whence $\mathcal P_{\fin} (\N_0)$ is fully elastic.

2. Let  $R = D[X_1, \ldots, X_n]$ be the polynomial ring  in $n \ge 2$ indeterminates over a domain $D$, and suppose that $\mathcal I (R)$ is a \BF-monoid. We consider the monoid
$\mathcal M(X_1,X_2)$  of all monomial ideals in the indeterminates $X_1,X_2$. This is a submonoid of $\mathcal I(R)$.
Recall that, for every $m \in \N_0$, we denote by $\mathcal M_{m;1,2}$  the set of all monomials of the form $X_1^rX_2^s$ with $r+s=m$ and $r,s \in \N_0$.
Let $\mathcal M_{2}(X_1,X_2) \subset \mathcal M (X_1,X_2)$  consist of all ideals
\[
\langle \mathcal N  \cup \{X_2^m\} \rangle \,,
\]
where $m \in \N_0$ and $\mathcal N \subset \mathcal M_{m;1,2}$ is any subset (note that, for example,
$\langle X_2 \rangle$ and $\langle X_1^3, X_1^2X_2, X_2^3 \rangle$ are in $\mathcal M_{2}(X_1,X_2)$, whereas $\langle X_1^3, X_1^2X_2 \rangle$ and $\langle X_1^2 \rangle$ do not belong to $\mathcal M_{2}(X_1,X_2)$).
Then  $\mathcal M_{2}(X_1,X_2) \subset \mathcal I(R)$ is a submonoid.

For $A \in \mathcal P_{\fin}(\N_0)$, say $A = \{m_1,\ldots, m_{\ell} \}$ with $\ell \geq 1$ and $0 \le m_1 < m_2 <\ldots < m_{\ell}$. We denote by
\[
I_A := \langle X_1^{m_{\ell} - m_1} X_2^{m_1}, X_1^{m_{\ell}-m_2}X_2^{m_2}, \ldots, X_1^{m_{\ell}-m_{\ell-1}} X_2^{m_{\ell-1}} ,X_2^{m_{\ell}} \rangle
\]
an ideal which clearly belongs to $\mathcal M_{2}(X_1,X_2)$.
Thus, we obtain a map
\[
\mathcal P_{\fin}(\N_0) \longrightarrow \mathcal M_{2}(X_1,X_2) \,, \quad \text{ given by } \quad A \longmapsto I_A \,,
\]
which is easily seen to be  a monoid isomorphism.
\end{proof}

\bigskip
\noindent
{\bf Acknowledgement.} We would like to thank  Rob Eggermont, Azhar Farooq, Florian Kainrath, Andreas Reinhart,  and Daniel Smertnig, for many helpful discussions. Furthermore, we would like to thank the reviewer for their careful work and all their suggestions.

\providecommand{\bysame}{\leavevmode\hbox to3em{\hrulefill}\thinspace}
\providecommand{\MR}{\relax\ifhmode\unskip\space\fi MR }
\providecommand{\MRhref}[2]{%
  \href{http://www.ams.org/mathscinet-getitem?mr=#1}{#2}
}
\providecommand{\href}[2]{#2}

\end{document}